\documentclass[letterpaper]{amsart}
\usepackage[utf8]{inputenc}
\usepackage{amsfonts}
\usepackage{amsmath}
\usepackage{amsthm}
\usepackage{amssymb}
\usepackage{comment}
\usepackage{multirow}
\usepackage{xspace}

\newcommand{\mf}[1]{\mathfrak{#1}}
\newcommand{\bb}[1]{\mathbb{#1}}
\newcommand{\mc}[1]{\mathcal{#1}}
\newcommand{\wt}[1]{\widetilde{#1}}
\newcommand{\mymag}{\text{mag}}
\newcommand{\dotarg}{ \ \cdot \ }

\newcommand{\bfX}{{\bf X}}
\newcommand{\bfZ}{{\bf Z}}

\DeclareMathOperator{\Exp}{exp}

\DeclareMathOperator{\Aut}{Aut}

\DeclareMathOperator{\so}{\mf{so}}

\DeclareMathOperator{\sgn}{sgn}
\DeclareMathOperator{\ad}{ad}
\DeclareMathOperator{\myspan}{span}

\DeclareMathOperator{\Per}{Per}

\newcounter{thmcounter}
\numberwithin{thmcounter}{section}

\newtheorem{thm}[thmcounter]{Theorem}
\newtheorem{lem}[thmcounter]{Lemma}
\newtheorem{cor}[thmcounter]{Corollary}
\newtheorem*{theorem*}{Theorem}

\theoremstyle{definition}
\newtheorem{defn}[thmcounter]{Definition}

\theoremstyle{remark}
\newtheorem{rmk}[thmcounter]{Remark}
\newtheorem{ex}[thmcounter]{Example}

\title{Periodic Magnetic Geodesics on Heisenberg Manifolds}
\author{Jonathan Epstein, Ruth Gornet, Maura B. Mast}
\date{\today}

\begin{document}

\maketitle

\begin{abstract}
We study the dynamics of magnetic flows on Heisenberg groups. Let $H$ denote the three-dimensional simply connected Heisenberg Lie group endowed with a left-invariant Riemannian metric and an exact, left-invariant magnetic field.  Let $\Gamma$ be a lattice subgroup of $H,$ so that $\Gamma\backslash H$ is a closed nilmanifold.   We first find an explicit  description of magnetic geodesics on $H$, then determine all closed magnetic geodesics and their lengths for $\Gamma \backslash H$. We then consider two applications of these results: the density of periodic magnetic geodesics and marked magnetic length spectrum rigidity. We show that tangent vectors to periodic magnetic geodesics are dense for sufficiently large energy levels. We also show that if $\Gamma_1, \Gamma_2 < H$ are two lattices such that $\Gamma_1 \backslash H$ and $\Gamma_2 \backslash H$ have the same marked magnetic length spectrum, then they are isometric as Riemannian manifolds. Both results show that this class of magnetic flows carries significant information about the underlying geometry. Finally, we provide an example to show that extending this analysis of magnetic flows  to the Heisenberg type setting is considerably more difficult.
\end{abstract}

\section{Introduction}

From the perspective of classical mechanics, the geodesics of a Riemannian manifold $(M,g)$ are the possible trajectories of a point mass moving in the absence of any forces and in zero potential. A magnetic field can be introduced by choosing a closed 2-form $\Omega$ on $M$. A charged particle moving on $M$ now experiences a Lorentz force, and its trajectory is called a magnetic geodesic. As with Riemannian geodesics, they can be handled collectively as a single object called the magnetic geodesic flow on $TM$ or $T^*M$ (see Section \ref{sec:mag flows} for precise definitions). Many classical questions concerning geodesic flows have corresponding analogs for magnetic flows. Indeed, magnetic flows  display a number of remarkable properties. See \cite{grognet1999maginnegcurv}, \cite{paternain2006horocycleflows}, \cite{burns2006rigidity}, \cite{butlerpaternain2008magnetic}, and \cite{abbondandolo2017exactflowsonsurfaces} for a sampling of results. 

One can interpret magnetic flows as a particular type of perturbation of the underlying geodesic flow. Much is known about the the underlying geodesic flow of nilmanifolds, and we are interested in what properties persist or fail to persist for magnetic flows. This perspective is adopted for the property of topological entropy in \cite{epstein2017topent} in the setting of two-step nilmanifolds and in \cite{butlerpaternain2008magnetic} in the setting of $\text{SOL}$ manifolds; and for topological entropy and the Anosov property in \cite{paternainpaternain1995}, \cite{paternainpaternain1997} and \cite{burnspaternain2002anosov}. In \cite{peyerimhoff2003} the authors show that at high enough energy levels the magnetic geodesics are quasi-geodesics with respect to the underlying Riemannian structure. An important classical question of geodesic flows concerns the existence of closed geodesics and related properties such as their lengths and their density. This paper focuses on these properties in the context of magnetic flows generated by left-invariant magnetic fields on Riemannian two-step nilmanifolds. Although this setting is more complicated than the Euclidean setting (i.e. 1-step nilmanifolds), many explicit computations are still tractable, and it has been a rich source of conjectures and counter-examples.

Let $H$ denote a simply connected $(2n+1)$-dimensional Heisenberg group endowed with a left-invariant Riemannian metric. The Lie group $H$ admits cocompact discrete subgroups (i.e. lattices) $\Gamma$ and, because the Riemannian metric is left-invariant, the quotient inherits a metric such that $\Gamma \backslash H$ is a compact Riemannian manifold and $H \to  \Gamma \backslash H$ is a Riemannian covering. A geodesic $\sigma(t)$ in $H$ is said to be translated by an element $\gamma \in H$ if $\gamma \sigma(t) = \sigma(t + \omega)$ for all $t$ and for some $\omega > 0$. A geodesic that is translated by $\gamma$ is said to be $\gamma$-periodic. When $\gamma \in \Gamma$, each geodesics translated by $\gamma$ will project to a smoothly closed geodesic in $\Gamma \backslash H$. Geodesic behavior in $\Gamma \backslash H$ and, more generally, in $\Gamma\backslash N$, where $N$ denotes a simply connected two-step nilpotent Lie group with a left-invariant metric, is fairly well understood. In the general Riemannian two-step case, it is possible to describe precisely the set of smoothly closed geodesics in $\Gamma \backslash N$, along with their lengths. See Eberlein \cite{eberlein1994geometry} for the Heisenberg case and Gornet-Mast \cite {gornetmast2000lengthspectrum} for the more general setting. Our main result is a complete analysis of left-invariant, exact magnetic flows on three-dimensional Heisenberg groups.

\begin{theorem*}[See Section \ref{sec:magnetic geodesic equations on simply connected 2n+1 dim Heis}, Lemma \ref{lem:effect of energy constraint on range of ells} and Theorem \ref{thm:periods of central elt in 3D Heisenberg}]
Let $H$ be a three-dimensional simply connected Heisenberg group, $g$ a left-invariant metric on $H$, and $\Omega$ a left-invariant, exact magnetic field on $H$. For any $\gamma \in H$, there is an explicit description of all the $\gamma$-periodic magnetic geodesics of the magnetic flow generated by $(H, g, \Omega)$ satisfying $\sigma(0) = e$, the identity element. The lengths of closed magnetic geodesics may be explicitly computed in terms of metric Lie algebra information.
\end{theorem*}
This theorem allows for the explicit computation of all closed magnetic geodesics in the free homotopy class determined by each $\gamma \in \Gamma$. Unlike the Riemannian case, closed magnetic geodesics exist in all nontrivial homotopy classes only for sufficiently large energy. In addition, there exist closed and contractible magnetic geodesics on sufficiently small energy levels. 


We give two applications of our main result. The first concerns the density of tangent vectors to closed magnetic geodesics. Eberlein analyzes this property for Riemannian geodesic flows on two-step nilmanifolds with a left-invariant metric, showing that for certain types of two-step nilpotent Lie groups (including Heisenberg groups), the vectors tangent to smoothly closed unit speed geodesics in the corresponding nilmanifold are dense in the unit tangent bundle \cite{eberlein1994geometry}; Mast \cite{mast1994closedgeods} and Lee-Park \cite{leepark1996density} broadened this result. In Theorem \ref{thm:density of closed mag geods}, we show that the density property continues to hold for magnetic flows on sufficiently high energy levels on the Heisenberg group. The second is a marked length spectrum rigidity result (see Section \ref{sec:rigidity} for the definition). It known that within certain classes of Riemannian manifolds, if two have the same marked length spectrum then they are isometric. This is true in the class of negatively curved surfaces (see \cite{croke1990rigidity} and \cite{otal1990commentarii, otal1990annals}) and compact flat manifolds (see \cite{berger1971lecturenotes}, \cite{berard1986spectralgeom}, \cite{miatellorossetti2003}). In \cite{grognet2005marked}, S. Grognet studies marked length spectrum rigidity of magnetic flows on surfaces with pinched negative curvature. In Theorem \ref{thm:MLS rigidity}, we show that the marked magnetic magnetic length spectrum of left-invariant magnetic systems on compact quotients of the Heisenberg group determine the Riemannian metric. Although it's a perturbation of geodesic flow, the magnetic flow still carries information about the underlying Riemannian manifold.

This paper is organized as follows. In Section \ref{sec:preliminaries}, we present the necessary preliminaries in order to state and prove the main theorems. The definition and basic properties of magnetic flows are given in Section \ref{sec:mag flows} and the necessary background on nilmanifolds is given in Section \ref{sec:geoemtry of two-step nilpotent lie groups}. Next, we show how a left-invariant Hamiltonian system on the cotangent bundle of a Lie group reduces to a so-called Euler flow on the dual to the Lie algebra. Such Hamiltonians are known as collective Hamiltonians, and this process is outlined in Section \ref{sec:left-inv hamiltonian on Lie groups}. Section \ref{sec:exact, left-inv mag forms} specializes the preceding to the case of exact, left-invariant magnetic flows on two-step nilpotent Lie groups. In Section \ref{sec:magnetic geodesic equations on simply connected 2n+1 dim Heis}, the magnetic geodesic equations on the $(2n+1)$-dimensional Heisenberg group are solved. In Section \ref{sec:compact quotients}, we apply these formulas to obtain our main theorem and the applications described above. Many geometric results for the Heisenberg group have been shown to hold for the larger class of Heisenberg type manifolds. In Section \ref{sec:HT manifolds}, we use a specific example to show why our analysis of magnetic flows on Heisenberg type manifolds is considerably more difficult. Lastly, the so-called $j$-maps are a central part of the theory of two-step Riemannian nilmanifolds. In the appendix, we provide an alternative approach to studying the magnetic geodesics using $j$-maps instead of collective Hamiltonians.

\section{Preliminaries} \label{sec:preliminaries}

\subsection{Magnetic flows} \label{sec:mag flows}

A \emph{magnetic structure} on a Riemannian manifold $\left(M,g\right)$
is a choice of closed $2$-form $\Omega$ on $M$, called the \emph{magnetic
$2$-form}. The \emph{magnetic flow} of $\left(M,g,\Omega\right)$
is the Hamiltonian flow $\Phi_{t}$ on $TM$ determined by the symplectic
form 
\begin{align} \label{eq:magnetic symplectic form}
\varpi_\mymag = \bar\varpi + \pi^*\Omega
\end{align}
and the kinetic energy Hamiltonian $H_0:TM\rightarrow \bb{R},$ given by
\begin{align}
H_0 \left( v \right)=\frac{1}{2}g\left(v,v\right) = \frac{1}{2}|v|^2 \text{.}
\end{align}
Here $\pi:TM \rightarrow M$ denotes the  canonical projection and
$\bar\varpi$ denotes the pullback via the Riemmanian metric of the canonical symplectic form
on $T^*M$. 

The magnetic flow models the motion of a charged particle under the effect of a magnetic field whose Lorentz force $F:TM\rightarrow TM$
is the bundle map defined via
\begin{align*}
\Omega_{x}\left(u,v\right)=g_{x}\left(F_{x}u,v\right)    
\end{align*}
for all $x\in M$ and all $u,v\in T_{x}M$. The orbits of the magnetic flow have the form $t\mapsto\dot{\sigma}\left(t\right)$, where $\sigma$ is a curve in $M$ such that 
\begin{equation} \label{eq:magnetic geodesic equation}
\nabla_{\dot{\sigma}}\dot{\sigma}=F\dot{\sigma}\text{.}
\end{equation}
In the case that $\Omega=0$, the magnetic flow reduces to Riemannian geodesic flow. A curve $\sigma$ that satisfies \eqref{eq:magnetic geodesic equation} is called a \emph{magnetic geodesic}. The physical interpretation of a magnetic geodesic is that it is the path followed by a particle with unit mass and charge under the influence of the magnetic field. Because $F$ is skew-symmetric, the acceleration of the magnetic geodesic is perpendicular to its velocity.

\begin{rmk}
It is straightforward to show that magnetic geodesics have constant
speed. In contrast to the Riemannian setting, a unit speed reparametrization
of a solution to \eqref{eq:magnetic geodesic equation} may no longer be a solution.
To see this, let $\sigma(s)$ be a solution that is not unit speed
and denote energy $E=\left\vert \dot{\sigma}\right\vert >0.$ Define
$\tau\left(s\right)=\sigma\left(s/E\right)$, which is unit speed.
Then 
\begin{align*}
\nabla_{\dot{\tau}}\dot{\tau}=\frac{1}{E^{2}}\nabla_{\dot{\sigma}}\dot{\sigma}=\frac{1}{E^{2}}F\dot{\sigma}=\frac{1}{E}F\dot{\tau}\neq F\dot{\tau}\text{,}    
\end{align*}
in general. Therefore, one views a magnetic geodesic as the path,
not the parameterized curve. (Observe that $\tau$ is a solution
to the magnetic flow determined by the magnetic form $\frac{1}{\sqrt{E}}\Omega$.)
\end{rmk}

Recall that the tangent and cotangent bundles of a Riemannian manifold are canonically identified, and the Riemannian metric on $TM \to M$ induces a non-degenerate, symmetric 2-tensor on $T^*M \to M$. We will present most of the theory in the setting of the cotangent bundle, while occasionally indicating how to translate to the tangent bundle.  Note that many authors use the tangent bundle approach. See for example \cite{burns2006rigidity}. 

Slightly abusing notation, we now let $\pi$ denote the basepoint map of the cotangent bundle, let $g$ denote the metric on the cotangent bundle, and define $H_0 : T^*M \to \bb{R}$ as $H_0(p) = \frac{1}{2} g(p,p) = \frac{1}{2} |p|^2$. Accordingly, the magnetic flow of $(M,g,\Omega)$ is the Hamiltonian flow $\Phi_t$ on the symplectic manifold $(T^*M, \varpi + \pi^* \Omega)$ determined by the Hamiltonian $H_0$. Regardless of  approach, the projections of the orbits to the base manifold will be the same magnetic geodesics determined by \eqref{eq:magnetic geodesic equation}.

On the cotangent bundle
\begin{align} \label{eq:cotangent magnetic symplectic form}
\varpi_\mymag = \varpi + \pi^*\Omega
\end{align}
defines a symplectic form as long as $\Omega$ is closed; $\Omega$ may be non-exact or exact. In the former case, $\Omega$ is referred to as a \emph{monopole}. In the latter case, when $\Omega$ is exact, the magnetic flow can be realized either as the Euler-Lagrange flow of an appropriate Lagrangian, or (via the Legendre transform) as a Hamiltonian flow on $T^\ast M$ endowed with its canonical symplectic structure. Note that even if two magnetic fields represent the same cohomology class, they generally determine distinct magnetic flows.

Suppose that $\Omega = d\theta$ for some 1-form $\theta$. A computation in local coordinates shows that the diffeomorphism $f:T^*M \to T^*M$ defined by $f(x,p) = (x,p-\theta_x)$ conjugates the Hamiltonian flow of $(T^*M, \varpi + \pi^{\ast}\Omega, H_0)$ with the Hamiltonian flow of $(T^*M, \varpi, H_1)$ where
\begin{align} \label{eq:modified Hamiltonian}
H_1(x,p) = \frac{1}{2}|p + \theta_x|^2.
\end{align}

\subsection{Example: Magnetic Geodesics in the Euclidean Plane} \label{sec:Euclidean plane example}

Before introducing two-step nilmanifolds in the following subsection, we first provide an example of a left-invariant magnetic system in a simpler context. 

Let $M = \bb{R}^2$ endowed with the standard Euclidean metric $g$. Let $\Omega = B \ dx \wedge dy$ denote a magnetic 2-form, where $(x,y)$ denote global coordinates and $B$ is a real parameter that can be interpreted as modulating the strength of the magnetic field. 

Let $\sigma_{v}\left(t\right) = (x(t),y(t))$ denote the magnetic geodesic through the identity $e=(0,0)$ with initial velocity $v = (x_0,y_0) = x_{0} \frac{\partial}{\partial x} + y_{0} \frac{\partial}{\partial y} \neq 0$ and energy $E=\sqrt{x_{0}^{2}+y_{0}^{2}}$. The Lorentz force $F$ satisfies $F \left( 1, 0 \right) = B (0,1)$ and $F \left( 0, 1 \right) = -B (1, 0)$. By \eqref{eq:magnetic geodesic equation} $\sigma_v(t)$ satisfies
\begin{align*}
    \left(\ddot{x},\ddot{y}\right)=F\left(\dot{x},\dot{y}\right)=B\left(-\dot{y},\dot{x}\right)\text{.}
\end{align*}
The unique solution satisfying $\sigma_v(0) = e$ and $\dot{\sigma}_v(0) = v$ is
\begin{align*}
x\left(t\right) & = -\frac{y_{0}}{B}\left(1-\cos\left(tB\right)\right)+\frac{x_{0}}{B}\sin\left(tB\right) \\
y\left(t\right) & = \frac{x_{0}}{B}\left(1-\cos\left(tB\right)\right)+\frac{y_{0}}{B}\sin\left(tB\right)\text{.}    
\end{align*}
Then $\sigma_{v}\left(t\right)$ is a \emph{circle} of radius $\frac{E}{|B|}$ and center $\left(-\frac{y_{0}}{B},\frac{x_{0}}{B}\right)$. It is immediate that magnetic geodesics cannot be reparameterized. For if $\sigma_{v'}(t)$ is another magnetic geodesic through the identity with $v'$ parallel to $v$ but with $|v| \neq |v'|$, then $\sigma_{v'}(t)$ will describe a circle of different radius. Furthermore magnetic geodesics are not even time-reversible. The magnetic geodesic $\sigma_{-v}\left(t\right)$ is a circle of radius $\frac{E}{|B|}$ and center $\left(\frac{y_{0}}{B},-\frac{x_{0}}{B}\right)$; in particular, $\sigma_{-v}\left(t\right)$ and $\sigma_{v}\left(t\right)$ are both circles of the same radius but trace different paths. Note that every magnetic geodesic in this setting is periodic. This will not be the case for two-step nilmanifolds.

\subsection{Left-invariant Hamiltonians on Lie groups} \label{sec:left-inv hamiltonian on Lie groups} Let $G$ be a Lie group with Lie algebra $\mf{g}$. On the one hand, $T^*G$ ( $=G\times \mathfrak{g}^*$) is a symplectic manifold and each function $H : T^*G \to \bb{R}$ generates a Hamiltonian flow with infinitesimal generator $X_H$. On the other hand, $\mf{g}^*$ is a Poisson manifold and each function $f : \mf{g}^* \to \bb{R}$ determines a derivation of $C^\infty(\mf{g}^*)$ and hence a vector field $E_f$, called the Euler vector field associated to $f$. When the function $H$ is left-invariant, i.e. $H((L_x)^* \alpha) = H(\alpha)$ for all $x \in G$ and all $\alpha \in T^*G$, it induces a function $h : \mf{g}^* \to \bb{R}$ and the flow of $X_H$ factors onto the flow of $E_h$. Moreover, the flow of $X_H$ can be reconstructed from $E_h$ and knowledge of the group structure of $G$. Note that this is a special case of a more general class of Hamiltonians, called collective Hamiltonians. More details and physical motivation can be found in Sections 28 and 29 of \cite{guillemin1990symplectic}. We outline below how we will use this approach to study magnetic flows. 

A Poisson manifold is a smooth manifold $M$ together with a Lie bracket $\{ \cdot, \cdot \}$ on the algebra $C^\infty(M)$ that also satisfies the property
\begin{align} \label{eq:Poisson Leibniz rule}
\{f, gh\} = \{f,g\}h + g\{f, h\}
\end{align}
for all $f,g,h \in C^\infty(M)$. Hence, for a fixed function $h \in C^\infty(M)$, the map $C^\infty(M) \to C^\infty(M)$ defined by $f \mapsto \{ f, h \}$ is a derivation of $C^\infty(M)$. Therefore, there is an Euler vector field $E_h$ on $M$ such that $E_h ( \cdot ) = \{ \cdot, h \}$. 

An important source of Poisson manifolds is the vector space dual to a Lie algebra. We will make use of the standard identifications $T_p \mf{g}^* \simeq \mf{g}^*$ and $T^*_p\mf{g}^* \simeq (\mf{g}^*)^* \simeq \mf{g}$, and $\langle \ \cdot \ , \ \cdot \ \rangle$ will denote the natural pairing between $\mf{g}$ and $\mf{g}^*$. For a function $f \in C^\infty(\mf{g}^*)$, its differential $df_p$ at $p \in \mathfrak{g}^*$ is identified with an element of the Lie algebra $\mf{g}$. The Lie bracket structure on $\mf{g}$ induces the Poisson structure on $\mf{g}^*$ by
\begin{align} \label{eq:Poisson structure on g^*}
\{ f, g \}(p) = - \langle p, \left[df_p, dg_p\right] \rangle= -p\left(\left[df_p,dg_p\right]\right).
\end{align}
Antisymmetry and the Jacobi Identity follow from the properties of the Lie bracket $[ \ \cdot \ , \ \cdot \ ]$, while the derivation property \eqref{eq:Poisson Leibniz rule} follows from the Leibniz rule for the exterior derivative.

It is useful to express the Euler vector field $E_h$ in terms of $h$ and the representation $\ad^* : \mf{g} \to \mf{gl}(\mf{g}^*)$ dual to the adjoint representation, defined as 
\begin{align} \label{eq:Lie algebra dual}
\langle \ad^*_X p, Y \rangle = - \langle p, \ad_X Y \rangle. 
\end{align}
From the definition of the differential of a function,
\begin{align*}
\langle E_h(p), df_p \rangle = E_h(f)(p) = \{ f, h \}(p) = - \langle p, [df_p, dh_p] \rangle = - \langle \ad^*_{dh_p} p, df_p \rangle.
\end{align*}
From this we conclude that
\begin{align} \label{eq:Euler v.f. on Lie algebra dual}
E_h(p) = - \ad^*_{dh_p} p.
\end{align}

Now consider $T^*G \simeq G \times \mf{g}^*$ trivialized via left-multiplication. Let $r: G \times \mf{g}^* \to \mf{g}^*$ be projection onto the second factor. If $h: \mf{g}^* \to \bb{R}$ is any smooth function, then $H = h \circ r$ is a left-invariant Hamiltonian on $T^*G$. Conversely, any left-invariant Hamiltonian $H$ factors as $H = h \circ r$. Recall that the canonical symplectic structure  $\varpi$ on $T^*G \simeq G \times \mf{g}^*$ is
\begin{align} \label{eq:canon symp structure on Lie group}
\varpi_{(x,p)} ( (U_1, \alpha_1), (U_2, \alpha_2) ) = \alpha_2 (U_1) - \alpha_1(U_2) + p([U_1, U_2])
\end{align}
where we identify $T_{(g,p)}T^*G \simeq \mf{g} \times \mf{g}^*$ (see section 4.3 of \cite{eberlein2004Cubo} for more details). To find an expression for the Hamiltonian vector field $X_H(x,p) = (X,\lambda)$ of a left-invariant Hamiltonian, first consider vectors of the form $(0,\alpha)$ in the equation $\varpi( X_H,\dotarg) = dH(\dotarg)$. We have
\begin{align*}
\varpi_{(x,p)} ( (X, \lambda), (0, \alpha) ) &= dH_{(x,p)}(0,\alpha) =d(h\circ r)_{(x,p)}(0,\alpha), \\
\alpha (X) - \lambda(0) + p([X, 0])  &= dh_p(\alpha), \\
\alpha (X) &= \alpha( dh_p ).
\end{align*}
Since this is true for all choices of $\alpha$, we get $X = dh_p$. Next consider vectors of the form $(U,0)$. Since $H$ is left-invariant,
\begin{align*}
\varpi_{(x,p)}((dh_p, \lambda), (U,0)) &= dH_{(x,p)}(U,0), \\
- \langle \lambda, U \rangle + \langle p, [dh_p, U] \rangle &= 0, \\
\langle \lambda, U \rangle &= - \langle \ad_{dh_p}^* p, U \rangle.
\end{align*}
Since this must be true for every $U$, we have that $\lambda = -\ad_{dh_p}^* p = E_h(p)$. For a left-invariant Hamiltonian, the equations of motions for its associated Hamiltonian flow are
\begin{align} \label{eq:equations of motion for left-inv ham}
X_H(x,p) = \begin{cases}
\dot{x} = (L_x)_*(dh_p) \\
\dot{p} = E_h(p) = - \ad^*_{dh_p} p
\end{cases}.
\end{align}

\subsection{The Geometry of Two-Step Nilpotent Metric Lie Groups} \label{sec:geoemtry of two-step nilpotent lie groups} Our objects of study in this paper are simply connected two-step nilpotent Lie groups endowed with a left-invariant metric.  For an excellent reference regarding  the geometry of these manifolds, see  \cite{eberlein1994geometry}. 

Let $\mf{g}$ denote a two-step nilpotent Lie algebra with Lie bracket $[\ ,\ ]$ and non-trivial center $\mf{z}$. That is, $\mf{g}$ is nonabelian and $[X,Y]\in\mf{z}$ for all $X,Y\in\mf{g}$. Let $G$ denote the unique, simply connected Lie group with Lie algebra $\mf{g}$; then $G$ is a two-step nilpotent Lie group. The Lie group exponential map $\exp : \mf{g} \to G$ is a diffeomorphism, with  inverse map   denoted by $\log : G \to \mf{g}$. Using the Campbell-Baker-Hausdorff formula, the multiplication law can be expressed as
\begin{align} \label{eq:multiplication law}
    \exp(X)\exp(Y) = \exp \left( X + Y + \frac{1}{2}[X,Y] \right).
\end{align}
For any $A \in \mf{g}$ and any $X \in T_A \mf{g} \simeq \mf{g}$, the push-forward of the Lie group exponential at $A$ is
\begin{align*}
    (\exp_*)_A (X) = (L_{\exp(A)})_*\left( X + \frac{1}{2}[X,A] \right).
\end{align*}
Using this, the tangent vector to any smooth path $\sigma(t) = \exp(U(t))$ in $G$ is given by
\begin{align} \label{eq:tangent vector field of path in two-step nilpotent group}
    \sigma'(t) =(L_{\sigma(t)})_* \left( U'(t) + \frac{1}{2}[U'(t),U(t)] \right).
\end{align}
When a two-step nilpotent Lie algebra $\mf{g}$ is endowed with an inner product $g$, then there is a natural decomposition $\mf{g} = \mf{v} \oplus \mf{z}$, where $\mf{z}$ is the center of $\mf{g}$ and $\mf{v}$ is the orthogonal complement to $\mf{z}$ in $\mf{g}$. Every central vector $Z \in \mf{z}$ determines a skew-symmetric linear transformation of $\mf{v}$ (relative to the restriction of $g$), denoted $j(Z)$, as follows:
\begin{align} \label{eq:j-map def}
    g(j(Z)V_1, V_2) = g([V_1, V_2],Z)
\end{align}
for any vectors $V_1, V_2 \in \mf{v}$. In fact, this correspondence is a linear map $j : \mf{z} \to \mf{so}(\mf{v})$. These maps, first introduced by Kaplan \cite{kaplan1981Cliffordmodules}, capture all of the geometry of a two-step nilpotent metric Lie group.  For example, the $j$-maps provide a very useful description of the Levi-Civita connection. For $V_1, V_2 \in \mf{v}$ and $Z_1, Z_2 \in \mf{z}$,
\begin{align}
    & \nabla_{X_1}X_2 = \frac{1}{2}[X_1, X_2], \nonumber \\ 
    & \nabla_{X_1}Z_1 = \nabla_{Z_1} X_1 = - \frac{1}{2}j(Z) X, \label{eq:Levi-Civita eqns} \\ 
    & \nabla_{Z_1}Z_2 = 0. \nonumber 
\end{align}

\subsection{Exact, Left-Invariant Magnetic Forms on Simply Connected Two-Step Nilpotent Lie Groups} \label{sec:exact, left-inv mag forms}

We use the formalism of  Subsection \ref{sec:left-inv hamiltonian on Lie groups} to express the equations of motion for the magnetic flow of an exact, left-invariant magnetic form on a simply connected two-step nilpotent Lie group. Throughout this section, $\mf{g}$ denotes a two-step nilpotent Lie algebra with an inner product and $G$ denotes the simply connected Lie group with Lie algebra $\mf{g}$ endowed with the left-invariant Riemannian metric determined by the inner product on $\mf{g}$. 

As a reminder, angled brackets denote the natural pairing of a vector space and its dual. Recall that any (finite dimensional) vector space $V$ is naturally identified with $V^{**}$ by sending any vector $v \in V$ to the linear functional $V^* \mapsto \bb{R}$ defined by evaluation on $v$. Using this identification, we can and do view elements of $V$ simultaneously as elements of $V^{**}$. The inner product on $\mf{g}^*$ is specified by a choice of linear map $\sharp: \mf{g}^* \to \mf{g}$ such that (a) $\langle p, \sharp(p) \rangle > 0$ for all $p \neq 0$ and (b) $\langle p, \sharp(q) \rangle = \langle \sharp(p), q \rangle$ for all $p,q \in \mf{g}^*$. The inner product of $p, q \in \mf{g}^*$ is then given by $\langle p, \sharp(q) \rangle$ and the induced norm is $|p| = \sqrt{\langle p, \sharp(p) \rangle}$. Conversely any inner product on $\mf{g}^*$ induces a map $\sharp : \mf{g}^* \to \mf{g}^{**} \simeq \mf{g}$ with the properties (a) and (b). Of course, $\sharp^{-1} = \flat$ is then the flat map and the inner product of $X$ and $Y$ in $\mf{g}$ can be computed as $\langle X, \flat(Y) \rangle$. 

Let $\mf{g} = \mf{v} \oplus \mf{z}$ be the decomposition of $\mathfrak{g}$ into the center and its orthogonal complement. Let $\mf{g}^* = \mf{v}^* \oplus \mf{z}^*$ be the corresponding decomposition where $\mf{v}^*$ is the set of functionals that vanish on $\mf{z}$ and vice versa.

\begin{lem} \label{lem:every exact left-invariant 2-form is d of something central}
If $\Omega$ is an exact, left-invariant 2-form on $G$, then there exists $B\in \bb{R}$ and $\zeta_m \in \mf{z}^*$ such that $|\zeta_m| = 1$ and $\Omega = d(B\zeta_m)$.
\end{lem}

\begin{proof}
By hypothesis, $\Omega = d\theta$ for some left-invariant 1-form $\theta$. By left-invariance, $\theta$ can be expressed as $\theta = \theta_\mf{v} + \theta_\mf{z}$, where $\theta_\mf{v} \in \mf{v}^*$ and $\theta_\mf{z} \in \mf{z}^*$, and  $d\theta_\mf{v}(X,Y) = -\theta_\mf{v}([X,Y]) $ for any $X,Y \in \mf{g}$. Because $[X,Y] \in \mf{z}$, $d\theta_\mf{v} = 0$. Hence
\begin{align*}
    \Omega = d\theta = d(\theta_\mf{v} + \theta_\mf{z}) = d\theta_\mf{z}.
\end{align*}
Lastly, set $\zeta_m = \theta_\mf{z} / |\theta_\mf{z}|$ and $B = |\theta_\mf{z}|$.
\end{proof}

Given $B\in\mathbb{R}$ and $\zeta_m\in\mf{g}^*$, we define the function $H:T^*G \to \bb{R}$ by
\begin{align} \label{eq:magnetic Hamiltonian}
    H(x,p) = \frac{1}{2}|p + B\zeta_m|^2.
\end{align}
By the previous lemma, we may assume $\zeta_m$ is  a unit element in $\mf{g}^*$ that vanishes on $\mf{v}$. Because $\zeta_m$ is left-invariant, $H$ is left-invariant and factors as $H = h \circ r$, where $h : \mf{g}^* \to \bb{R}$ is the function
\begin{align} \label{eq:reduced magnetic Hamiltonian}
    h(p) = \frac{1}{2}|p + B\zeta_m|^2.
\end{align}
Note that when $B = 0$, the Hamiltonian flow of $H$ is the geodesic flow of the chosen Riemannian metric.

\begin{lem} \label{lem:differential of magnetic hamiltonian}
The differential of $h$ is $dh_p = \sharp(p + B \zeta_m)$.
\end{lem}

\begin{proof}
For any $p \in \mf{g}$ and any $q \in T_p\mf{g}^* \simeq \mf{g}^*$, we compute
\begin{align*}
    \langle q, dh_p \rangle &= \frac{d}{dt}\bigg|_{t = 0} h(p + tq) \\
    &= \frac{1}{2} \frac{d}{dt}\bigg|_{t = 0} |p + tq + B\zeta_m |^2 \\
    &= \frac{1}{2} \frac{d}{dt}\bigg|_{t = 0} \langle p + B\zeta_m + tq, \sharp(p + B\zeta_m + tq) \rangle \\
    &= \frac{1}{2} \frac{d}{dt}\bigg|_{t = 0} \left( | p + B\zeta_m |^2 + 2t\langle p + B\zeta_m , \sharp(q) \rangle + t^2 |q|^2 \right) \\
    &= \langle p + B\zeta_m, \sharp(q) \rangle.
\end{align*}
The Lemma now follows from the properties of $\sharp$.
\end{proof}

We now prove that the Euler vector field on $\mf{g}^*$ is independent of the choice of exact magnetic field, including the choice $\Omega = 0$.

\begin{lem} \label{lem:independent of B}
Let $h \in C^\infty(\mf{g}^*)$ be any function of the form \eqref{eq:reduced magnetic Hamiltonian} and define the function $h_0 \in C^\infty(M)$ by $h_0(p) = \frac{1}{2}|p|^2$. Then $E_{h_0} = E_h$.
\end{lem}

\begin{proof}
For any $\zeta \in \mf{z}^*$ and any $V \in \mf{v}$, $\langle V, \flat(\sharp(\zeta)) \rangle = \langle V, \zeta \rangle = 0$ shows that $\sharp(\mf{z}^*) = \mf{z}$. For any $X \in \mf{g}$,
by the previous lemma,
\begin{align*}
    \langle \ad^*_{dh_p} p, X \rangle = - \langle p, [\sharp(p + B\zeta_m), X] \rangle = - \langle p, [\sharp p, X] \rangle = \langle \ad^*_{(dh_0)_p} p, X \rangle.
\end{align*}
Hence $\ad^*_{dh_p} = \ad^*_{(dh_0)_p}$ and the proof follows from the expression \eqref{eq:Euler v.f. on Lie algebra dual} for the Euler vector field.
\end{proof}

We now describe the structure of the Euler vector field. Much of this can be gleaned from the results of \cite{eberlein1994geometry}. However, we include it here for the sake of self-containment. For any $X \in \mf{g}$ and $p \in \mf{g}^*$, we write $X = X_\mf{v} + X_\mf{z}$ and $p = p_\mf{v} + p_\mf{z}$ for the respective orthogonal decomposition according to $\mf{g} = \mf{v} \oplus \mf{z}$ and $\mf{g}^* = \mf{v}^* \oplus \mf{z}^*$. 

\begin{lem} \label{lem:structure of Euler vector field on two-step}
The integral curves of the Euler vector field $E_h$ are of the form $p(t) = p_\mf{v}(t) + \zeta_0$ where $\zeta_0 \in \mf{z}^*$ and $p_\mf{v}(t) \in \mf{v}^*$ is a path that satisfies $p_\mf{v}'(t) = A(p_\mf{v}(t))$ for some skew-symmetric transformation of $\mf{v}^*$.
\end{lem}

\begin{proof}
From \eqref{eq:Lie algebra dual},  the dual adjoint representation clearly has the following properties: $\ad^*_{Z} = 0$ for every $Z \in \mf{z}$, $\ad^*_X(\mf{g}^*) \subset \mf{v}^*$ for all $X \in \mf{g}$, and $\ad^*_X(\mf{v}^*) = \{ 0 \}$ for every $X \in \mf{g}$. From this, if $p(t) = p_\mf{v}(t) + p_\mf{z}(t)$ is an integral curve of $E_h$, then $p_\mf{z}(t) = p_\mf{z}(0) = \zeta_0$ is constant, and, using Lemmas \ref{lem:differential of magnetic hamiltonian} and  \ref{lem:independent of B}, $p_{\mf{v}}(t)$ must satisfy the system
\begin{align*}
    p'_{\mf{v}}(t) = E_h(p(t)) = -\ad^*_{dh_{p(t)}} p(t) = -\ad^*_{\sharp(p_{\mf{v}}(t))} p_\mf{z}(t) = -\ad^*_{\sharp(p_{\mf{v}}(t))} \zeta_0.
\end{align*}
Since $A:\mf{v}^* \to \mf{v}^*$ is skew-symmetric with respect to the inner product restricted to $\mf{v}^*$, this completes the Lemma.
\end{proof}

Let $(G,g,\Omega)$ be a magnetic system, where $G$ is a simply connected two-step nilpotent Lie group, $g$ is a left-invariant metric, and $\Omega$ an exact, left-invariant magnetic form. Let $\flat : \mf{g} \to \mf{g}^*$ and $\sharp = \flat^{-1}$ be the associated flat and sharp maps, and let $\zeta_m$ be as in Lemma \ref{lem:every exact left-invariant 2-form is d of something central}. The magnetic flow can be found as follows. First, compute the coadjoint representation of $\ad^* : \mf{g} \to \mf{gl}(\mf{g}^*)$ and integrate the vector field $E(p) = -\ad_{dh_p}^* p$. It follows that the curves $\sigma(t)$ satisfying $\sigma'(t) = dh_{p(t)}$, where $p(t)$ is an integral curve of $E$, will be magnetic geodesics. To make this step more explicit, let $\mf{g} = \mf{v} \oplus \mf{z}$ be the decomposition of $\mf{g}$ where $\mf{z}$ is the center and $\mf{v}$ is its orthogonal complement. Suppose that $p(t) = p_1(t) + \zeta_0$ is an integral curve of $E$, where $p_1(t) \in \mf{v}^*$ and $\zeta_0 \in \mf{z}^*$, and $\sigma(t) = \exp(\bfX(t) + \bfZ(t))$ is a path in $G$, where $\bfX(t) \in \mf{v}$ and $\bfZ(t) \in \mf{z}$. Using \eqref{eq:tangent vector field of path in two-step nilpotent group}, we can decompose the equation $\sigma'(t) = dh_{p(t)} = \sharp(p(t) + B\zeta_m)$ as
\begin{align} 
    & \bfX'(t) = \sharp(p_1(t)), \label{eq:mag goed equations v2a} \\
    & \bfZ'(t) + \frac{1}{2}[\bfX'(t),\bfX(t)] = \sharp(\zeta_0 + B \zeta_m). \label{eq:mag goed equations v2b}
\end{align}
Assuming that the path satisfies $\sigma(0) = e$, the first equation can be integrated to find $\bfX(t)$, which then allows the second equation to be integrated to find $\bfZ(t)$.

\begin{rmk}
The presence of the magnetic field can be thought of as a perturbation of the geodesic flow of $(G,g)$, modulated by the parameter $B$. In the procedure outlined here for two-step nilpotent Lie groups, the magnetic field only appears in the final step. The Euler vector field, and hence its integral curve, is unchanged by the magnetic field. In addition, the non-central component of the magnetic geodesics is the same as that of the Riemannian geodesics. The presence of a left-invariant exact magnetic field only perturbs the geodesic flow in central component of the Riemannian geodesics.
\end{rmk}

\begin{rmk} \label{rmk:how to calculate energy of mag geod from initial conditions}
For a magnetic geodesic $\sigma(t)$, we will call $|\sigma'(t)|$ its \emph{energy}. Note that this is a conserved quantity for magnetic flows. Since we are not considering a potential, the total energy of a charged particle in a magnetic system is its kinetic energy $|\sigma'(t)|^2 /2$. Although this would be commonly referred to as the energy in the physics and dynamics literature, we find our convention to be more convenient from our geometric viewpoint.
\end{rmk}

\begin{rmk} \label{rmk:general expression for energy squared}
Although $t \mapsto (\sigma(t), p(t))$ is an integral curve of the Hamiltonian vector field, the Hamiltonian $h$ is not the kinetic energy, and hence the energy of the magnetic geodesic is not equal to $|p(0)|$. Instead, by \eqref{eq:mag goed equations v2a} and \eqref{eq:mag goed equations v2b}, the energy squared is
\begin{align} \label{eq:energy of mag geod}
    |\sigma'(t)|^2 = |\sharp(p(0)) + B \sharp(\zeta_m)|^2 = |\sharp(p_1(0))|^2 + |\sharp(\zeta_0 + B \zeta_m)|^2.
\end{align}
\end{rmk}

\section{Simply Connected $(2n + 1)$-Dimensional Heisenberg Groups}

\label{sec:magnetic geodesic equations on simply connected 2n+1 dim Heis}

Let $\mf{h}_{n} = \myspan\{ X_1, \ldots, X_n, Y_1, \ldots, Y_n, Z\}$ and define a bracket structure on $\mf{h}_n$ by declaring the only nonzero brackets among the basis vectors to be $[X_i, Y_i] = Z$ and extending $[ \dotarg, \dotarg ]$ to all of $\mf{h}_n \times \mf{h}_n$ by bilinearity and skew-symmetry. Then $\mf{h}_n$ is a two-step nilpotent Lie algebra called the Heisenberg Lie algebra of dimension $2n + 1$, and the simply connected Lie group $H_n$ with Lie algebra $\mf{h}_n$ is called the Heisenberg group of dimension $2n+1$. Let $\{ \alpha_1, \beta_1, \ldots, \alpha_n, \beta_n, \zeta \}$ be the dual basis of $\mf{h}_n^*$. The following Lemma, proven in Lemma 3.5 of \cite{gordon1986spectrum}, shows that to consider every inner product on $\mf{h}_n$, we need only consider inner products on $\mf{h}_n$ that have a simple relationship to the bracket structure.

\begin{lem}
Let $g$ be any inner product on $\mf{h}_n$. There exists $\varphi \in \Aut(\mf{h}_n)$ such that
\begin{align} \label{eq:standard orthonomral basis for Heisenberg algebra}
    \left\{ \frac{X_1}{\sqrt{A_1}}, \ldots, \frac{X_n}{\sqrt{A_n}}, \frac{Y_1}{\sqrt{A_1}}, \ldots, \frac{Y_n}{\sqrt{A_n}}, Z \right\}
\end{align}
is an orthonormal basis relative to $\varphi^*g$, where $A_i > 0,$ $i=1\dots n,$ are positive real numbers.
\end{lem}

\begin{proof}
Consider the linear map defined by
\begin{align*}
    X_i \mapsto \frac{X_i}{|Z|} \qquad Y_i \mapsto Y_i \qquad Z \mapsto \frac{Z}{|Z|}.
\end{align*}
This is an automorphism of $\mf{h}_n$ and $Z$ is a unit vector relative to the pullback of the metric. Hence we can and will assume that $|Z| = 1$.

Let $\psi_1$ be the linear map defined by $\psi_1(X_i) = \allowbreak X_i - g(X_i,Z) Z$, $\psi_1(Y_i) = Y_i - g(Y_i,Z)Z$, and $\psi_1(Z) = Z$. Now $\psi_1 \in \Aut(\mf{h}_n)$ and $\mf{v} = \myspan\{ X_1, \ldots, X_n, \allowbreak Y_1, \ldots, Y_n \}$ is orthogonal to $\mf{z} = \myspan\{ Z \}$ relative to $\psi_1^* g$.

Next consider the map $j(Z) \in \so(\mf{v}, \psi_1^*g)$. Because it is skew-symmetric, there exists a $\psi_1^*g$-orthonormal basis $\{ \wt{X}_1, \ldots, \wt{X}_n, \wt{Y}_1, \ldots, \wt{Y}_n \}$ of $\mf{v}$ such that $j(Z)\wt{X}_i = d_i \wt{Y}_i$ and $j(Z)\wt{Y}_i = -d_i\wt{X}_i$
for some real numbers $d_i > 0$. Because
\begin{align*}
    (\psi_1^*g)(Z, [\wt{X}_i, \wt{Y}_i]) = (\psi_1^*g)(j(Z)\wt{X}_i, \wt{Y}_i) = (\psi_1^*g)(d_i\wt{Y}_i, \wt{Y}_i) = d_i
\end{align*}
we see that $[\wt{X}_i, \wt{Y}_i] = d_i Z$. Define the linear map $\psi_2$ by
\begin{align*}
    \psi_2(X_i) = \frac{1}{\sqrt{d_i}} \wt{X}_i \qquad \psi_2(Y_i) = \frac{1}{\sqrt{d_i}} \wt{Y}_i \qquad \psi_2(Z) = Z.
\end{align*}
Then $\psi_2 \in \Aut(\mf{h}_n)$ because
\begin{align*}
    [\psi_2(X_i),\psi_2(Y_i)]  =  Z = \psi_2(Z)=  \psi_2(\left[ X_i, Y_i \right])
\end{align*}
and, setting $A_i = d_i$, it is clear that the  basis \eqref{eq:standard orthonomral basis for Heisenberg algebra} is orthonormal relative to $\psi_2^* (\psi_1^* g)$. Hence $\varphi = \psi_1 \circ \psi_2$ is the desired automorphism of $\mf{h}_n$.
\end{proof}

When \eqref{eq:standard orthonomral basis for Heisenberg algebra} is an orthonormal basis of $\mf{h}_n$, the sharp and flat maps are given by
\begin{center}
\begin{tabular}{l l}
    $\flat( X_i / \sqrt{A_i} ) = \sqrt{A_i} \alpha_i,$ & $\sharp(\sqrt{A_i}\alpha_i) = X_i / \sqrt{A_i},$ \\
    $\flat( Y_i / \sqrt{A_i} ) = \sqrt{A_i} \beta_i,$ & $\sharp(\sqrt{A_i}\beta_i) = Y_i / \sqrt{A_i},$ \\
    $\flat( Z ) = \zeta,$ & $\sharp(\zeta) = Z.$
\end{tabular}
\end{center}
Relative to the basis $\{ X_1, \ldots, X_n, Y_1, \ldots, Y_n, Z\}$, the adjoint representation is
\begin{align*}
    \ad_U = \begin{bmatrix}
    0 & \cdots &  &  &  & 0 & 0 \\
    \vdots &  &  &  &  & \vdots & \vdots \\
    0 & \cdots &  &  &  & 0 & 0 \\
    -y_1 & \cdots & -y_n & x_1 & \cdots & x_n & 0
    \end{bmatrix}
\end{align*}
where $U = \sum x_i X_i + \sum b_i y_i + z Z$. Relative to the dual basis, the coadjoint representation is the negative transpose
\begin{align*}
    \ad^*_U = -(\ad_U)^T = \begin{bmatrix}
    0 & \cdots &  &  &  & 0 & y_1 \\
    \vdots &  &  &  &  & \vdots & \vdots \\
    0 & \cdots &  &  &  & 0 & y_n \\
    0 & \cdots &  &  &  & 0 & -x_1 \\
    \vdots &  &  &  &  & \vdots & \vdots \\
    0 & \cdots &  &  &  & 0 & -x_n \\
    0 & \cdots &  &  &  & 0 & 0
    \end{bmatrix}.
\end{align*}
Because the center of $\mf{h}_n$ is one-dimensional, $\zeta_m = \zeta$, where $\zeta_m$ is as specified in Lemma \ref{lem:every exact left-invariant 2-form is d of something central}. 
Letting $p = \sum_i a_i \alpha_i + \sum_i b_i \beta_i + c\zeta$ be a point in $\mf{h}_n^*$, the differential of the Hamiltonian is
\begin{align*}
    dh_p = \sharp(p + B\zeta) = \sum_i \frac{a_i}{A_i}X_i + \sum_i \frac{b_i}{A_i}Y_i + (c+ B)Z
\end{align*}
and the Euler vector field is
\begin{align*}
    E_h(p) = -\ad_{dh_p}^* p = \sum_i \frac{-cb_i}{A_i} \alpha_i + \sum_i \frac{ca_i}{A_i} \beta_i.
\end{align*}
To integrate the system $p' = E_h(p)$, note that the central component of the Euler vector field is constant by Lemma \ref{lem:structure of Euler vector field on two-step}. 
Suppose that $p(t) = \sum a_i(t) \alpha_i + \sum b_i(t) \beta_i + c(t) \zeta$ is a solution that satisfies the initial condition $p(0) = \sum u_i \alpha_i + \sum v_i \beta_i + z_0 \zeta$. Then $c(t) = z_0$ and the remaining components form a linear system,
\begin{align*}
    a_i'(t) = -\frac{z_0}{A_i}b_i(t) \qquad \qquad b_i'(t) = \frac{z_0}{A_i} a_i(t)
\end{align*}
that is directly integrated to find
\begin{align*}
    a_i(t) &= u_i \cos \left( \frac{z_0 t}{A_i} \right) - v_i \sin \left( \frac{z_0 t}{A_i} \right), \\
    b_i(t) &= u_i \sin \left( \frac{z_0 t}{A_i} \right) + v_i \cos \left( \frac{z_0 t}{A_i} \right).
\end{align*}

With  an expression for the integral curves of the Euler vector field now established, we use equations \eqref{eq:mag goed equations v2a} and \eqref{eq:mag goed equations v2b} to obtain a coordinate expression for the magnetic geodesics through the identity. Let $\bfX(t) = \sum x_i(t) X_i + \sum y_i(t) Y_i$. If $z_0 \neq 0$, a direct integration of \eqref{eq:mag goed equations v2a} together with $\bfX(0) = 0$ yields
\begin{align*}
    x_i(t) &= \frac{u_i}{z_0}\sin \left( \frac{z_0 t}{A_i} \right) + \frac{v_i}{z_0}\cos \left( \frac{z_0 t}{A_i} \right) - \frac{v_i}{z_0},\\
    y_i(t) &= -\frac{u_i}{z_0} \cos \left( \frac{z_0 t}{A_i} \right) + \frac{v_i}{z_0}\sin \left( \frac{z_0 t}{A_i} \right)  + \frac{u_i}{z_0}.
\end{align*}  
If $z_0 = 0,$ we obtain
\begin{align*}
    x_i(t) &= \frac{u_i}{A_i}t, \\
    y_i(t) &= \frac{v_i}{A_i}t.
\end{align*}
Because the center is one-dimensional, the central component $\bfZ(t)$ in \eqref{eq:mag goed equations v2b} can be expressed as $\bfZ(t) = z(t)Z$. To integrate \eqref{eq:mag goed equations v2b} in the case that $z_0 \neq 0$, first compute
\begin{align*}
    [\bfX'(t), \bfX(t)] = \sum\left(x'_iy_i-x_iy'_i\right)Z=\sum \frac{u_i^2 + v_i^2}{A_i z_0} \left( \cos \left(\frac{z_0 t}{A_i} \right) - 1 \right) Z
\end{align*}
so that
\begin{align*}
    \bfZ'(t) = z'(t)Z &= \sharp(z_0 \zeta + B \zeta) - \frac{1}{2}[\bfX'(t), \bfX(t)] \\
    &= (z_0 + B)Z - \sum \frac{u_i^2 + v_i^2}{2 A_i z_0} \left( \cos \left(\frac{z_0 t}{A_i} \right) - 1 \right) Z \\
    &= \left(z_0 + B + \sum \frac{u_i^2 + v_i^2}{2 A_i z_0} \right)Z - \sum \frac{u_i^2 + v_i^2}{2 A_i z_0} \cos \left(\frac{z_0 t}{A_i} \right) Z
\end{align*}
and hence
\begin{align*}
    z(t) &= \left( z_0 + B + \sum \frac{u_i^2 + v_i^2}{2 A_i z_0} \right)t - \sum \frac{u_i^2 + v_i^2}{2z_0^2} \sin \left(\frac{z_0 t}{A_i} \right).
\end{align*}
In summary, when $z_0 \neq 0$, every magnetic geodesic $\sigma(t) = \exp( \sum x_i(t) X_i + \sum y_i(t) Y_i + z(t)Z)$ satisfying $\sigma(0) = e$ has the form
\begin{align}
    x_i(t) &= \frac{u_i}{z_0}\sin \left( \frac{z_0 t}{A_i} \right) - \frac{v_i}{z_0} \left(1- \cos \left( \frac{z_0 t}{A_i} \right)  \right), \label{eq:x_i component of mag geod} \\
    y_i(t) &= \frac{u_i}{z_0} \left( 1 - \cos \left( \frac{z_0 t}{A_i} \right) \right) + \frac{v_i}{z_0}\sin \left( \frac{z_0 t}{A_i} \right), \label{eq:y_i component of mag geod} \\
    z(t) &= \left( z_0 + B + \sum \frac{u_i^2 + v_i^2}{2 A_i z_0} \right)t - \sum \frac{u_i^2 + v_i^2}{2z_0^2} \sin \left(\frac{z_0 t}{A_i} \right). \label{eq:z component of mag geod}
\end{align}
When $z_0 = 0,$ we obtain
\begin{align}
    x_i(t) &= \frac{u_i}{A_i}t, \label{eq:x_i component of nonspiraling mag geod} \\
    y_i(t) &= \frac{v_i}{A_i}t, \label{eq:y_i component of nonspiraling mag geod} \\
    z_i(t) &= B t. \label{eq:z component of nonspiraling mag geod} 
\end{align}

\begin{rmk} \label{rmk:spiraling terminology}
A magnetic geodesic $\sigma(t)$ will be a one-parameter subgroup if and only if $z_0 = 0$ or $z_0 \neq 0$ and $u_i = v_i = 0$ for all $i$. We will sometimes call a magnetic geodesic \emph{spiraling} if it is not a one-parameter subgroup, and \emph{non-spiraling} if it is. We will also call a magnetic geodesic \emph{central} if it is of the form $\sigma(t) \in Z(H_n)$ for all $t$.
\end{rmk}
The initial velocity of the magnetic geodesic $\sigma(t)$ is
\begin{align*}
    \sigma'(0) = \sum \left( \frac{u_i}{A_i} X_i + \frac{v_i}{A_i} Y_i \right) + (z_0 + B)Z.
\end{align*}
Because $|X_i|^2 = |Y_i|^2 = A_i$, we can compute the square of the energy $E = |\sigma'(t)| = |\sigma'(0)|$ as (see Remark \ref{rmk:how to calculate energy of mag geod from initial conditions}) 
\begin{align} \label{eq:energy of spiraling mag geod}
    E^2 = |\sigma'(0)|^2 = \sum_i  \frac{u_i^2 + v_i^2}{A_i} + (z_0 + B)^2
\end{align}
Note that this expression is valid for all values of $z_0$.

\begin{thm} \label{thm:closed contractible magnetic geodesics}
There exist periodic magnetic geodesics with energy $E$ if and only if $0 < E < |B|$. For any $0 < E < |B|$, let $z_0 = -\sgn(B)\sqrt{B^2 - E^2}$ and let $u_i$ and $v_i$ be any numbers satisfying \eqref{eq:energy of spiraling mag geod}. Then the spiraling magnetic geodesics determined by $u_1,v_1, \ldots, u_n, v_n, z_0$ will be periodic of energy $E$. Moreover, the period of such a geodesic is $\omega = 2\pi A / z_0$.
\end{thm}

\begin{proof}
Recall that non-spiraling magnetic geodesics cannot be periodic. Inspection of the coordinate functions \eqref{eq:x_i component of mag geod}-\eqref{eq:z component of mag geod} of a spiraling magnetic geodesic shows they will yield a periodic magnetic geodesic if and only if the coefficient of $t$ in \eqref{eq:z component of mag geod} is zero. This condition is
\begin{align*}
    0 &= z_0 + B + \sum \frac{u_i^2 + v_i^2}{2 A_i z_0} = z_0 + B + \frac{1}{2z_0} \left( E^2 - (z_0 + B)^2 \right)
\end{align*}
or
\begin{align*}
    z_0^2 = B^2 - E^2.
\end{align*}
It can only be satisfied when $E < |B|$. To obtain a spiraling magnetic geodesic we need to require that $(z_0 + B)^2 < E^2$ or, equivalently, $z_0 \in (-B - E, -B + E)$. Since this interval contains only negative or positive numbers, depending on the sign of $B$, we must choose $z_0 = -\sgn(B)\sqrt{B^2 - E^2}$. Finally, to see that $z_0$ is indeed contained in this interval, note that $\sqrt{(B-E)(B+E)} = \sqrt{(-B+E)(-B-E)}$ is the geometric mean of the endpoints of interval.
\end{proof}

\begin{ex} \label{ex:mag geods on 3D Heis}
For convenience we state  the component functions of a magnetic geodesic $\sigma(t) = \exp(x(t)X + y(t)Y + z(t)Z)$ in the 3-dimensional Heisenberg group (i.e. $n = 1$) with $\sigma(0) = e$. To ease notation, we use the dual bases $\{ \alpha, \beta, \zeta\}$ and $\{X,Y,Z\}$ for $\mf{h}_1^*$ and $\mf{h}_1$, respectively, and we let $A = A_1$. Given a point $p(0) = u_0 \alpha + v_0 \beta + z_0 \zeta$,  $z_0 \neq 0$, the corresponding magnetic geodesic has component functions
\begin{align*}
    x(t) &= \frac{u_0}{z_0}\sin \left( \frac{z_0 t}{A} \right) - \frac{v_0}{z_0} \left(1- \cos \left( \frac{z_0 t}{A} \right)  \right), \\
    y(t) &= \frac{u_0}{z_0} \left( 1 - \cos \left( \frac{z_0 t}{A} \right) \right) + \frac{v_0}{z_0}\sin \left( \frac{z_0 t}{A} \right), \\
    z(t) &= \left( z_0 + B + \frac{u_0^2 + v_0^2}{2 A z_0} \right)t - \frac{u_0^2 + v_0^2}{2z_0^2} \sin \left(\frac{z_0 t}{A} \right). 
\end{align*}
When $z_0 = 0,$ we obtain
\begin{align*}
    x(t) = \frac{u_0}{A}t \qquad y(t) = \frac{v_0}{A}t \qquad z(t) = B t.
\end{align*}
\end{ex}

\begin{rmk}
\label{rmk:RiemannVsMagGeods}
It is instructive to compare the magnetic geodesics on $\bb{R}^2$ given in Section \ref{sec:Euclidean plane example} and the magnetic geodesics on $H_1$ given in Example \ref{ex:mag geods on 3D Heis}. In the former, all magnetic geodesics are closed circles with radii that depend on the energy. In the latter, the paths $x(t)X + y(t)Y$ through the complement to the center are also circles whose radii depend on both the energy and $z_0$.  It is also worth noting some qualitative differences between Riemannian geodesics  and magnetic geodesics on Heisenberg groups. In Riemannian case, one-parameter subgroups of the form $\exp(t(x_0 X + y_0 Y))$ are always geodesics. In contrast, the central component $z(t)$ of a magnetic geodesic can never be zero. Finally, note that in the Riemannian setting there are never closed geodesics in $H_n$ (compare with  Theorem \ref{thm:closed contractible magnetic geodesics}).
\end{rmk}

\section{Compact Quotients of Heisenberg Groups} \label{sec:compact quotients}

A geodesic $\sigma: \bb{R} \to M$ in a Riemannian manifold $M$ is called periodic or (smoothly) closed if $\sigma(t + \omega) = \sigma(t)$ for all $t \in \bb{R}$. A periodic or closed magnetic geodesic is defined similarly, and we now investigate the closed magnetic geodesics on manifolds of the form $\Gamma\backslash H_n$, where $\Gamma$ is a cocompact (i.e., $\Gamma\backslash H_n$ compact), discrete subgroup of the $(2n+1)$-dimensional simply connected Heisenberg group $H_n$. As is common, we proceed by considering $\gamma$-periodic magnetic geodesics on the universal cover $H_n$. An important distinction between the magnetic and Riemannian settings is that in the latter one needs to address each energy level separately because magnetic geodesics cannot be reparameterized.

\subsection{$\gamma$-Periodic Magnetic Geodesics}

\begin{defn}
Let $N$ be a simply connected nilpotent Lie group with left invariant metric and magnetic form. For any $\gamma \in N$ not equal to the identity, a magnetic geodesic $\sigma\left(t\right)$ is called \emph{$\gamma$-periodic with period $\omega$} if $\omega \neq 0$ and for all $t\in\mathbb{R}$ 
\begin{align} \label{eq:gamma-periodic geodeisc condition}
    \gamma\sigma\left(t\right)=\sigma\left(t+\omega\right).
\end{align}
We also say that \emph{$\gamma$ translates the magnetic geodesic $\sigma(t)$ by amount $\omega$}. The number $\omega$ is called a \emph{period of $\gamma$}.
\end{defn}
When $\Gamma < N$ is a cocompact discrete subgroup and $\gamma \in \Gamma$, a $\gamma$-periodic magnetic geodesic will project to a smoothly closed magnetic geodesic under the mapping $N \rightarrow \Gamma\backslash N$ and will be contained in the free homotopy class represented by $\gamma$.  Every periodic magnetic geodesic on $\Gamma\backslash N$ arises as the image of a $\gamma$-periodic magnetic geodesic on $N$.  

\begin{lem} \label{lem:non-central elts only translate non-spiraling mag goeds}
Let $\gamma = \exp(V_\gamma + Z_\gamma) \in H_n$, where $Z_\gamma \in Z(\mf{h}_n)$ and $V_\gamma$ is orthogonal to $Z(\mf{h}_n)$, and let $\sigma(t) = \exp(\bfX(t) + \bfZ(t))$ be a $\gamma$-periodic magnetic geodesic. If $V_\gamma \neq 0$, then $\sigma$ is a noncentral 1-parameter subgroup (see Remark \ref{rmk:spiraling terminology}).
\end{lem}

\begin{proof}
Repeated use of \eqref{eq:gamma-periodic geodeisc condition} shows that $\gamma^k \sigma(t) = \sigma(t + k\omega)$. Using the multiplication formula \eqref{eq:multiplication law} on each side of the equation, the non-central components must satisfy $kV_\gamma + \bfX(t) = \bfX(t + \omega)$. If $V_\gamma \neq 0$, then the vector-valued function $\bfX(t + \omega) - \bfX(t)$ must be unbounded. Inspection of the magnetic geodesic equation \eqref{eq:x_i component of mag geod}-\eqref{eq:z component of nonspiraling mag geod} shows that this can only happen if $z_0 = 0$, i.e. $\sigma$ is a 1-parameter subgroup. Moreover $\sigma$ cannot be a central 1-parameter subgroup because then the left-hand side of \eqref{eq:gamma-periodic geodeisc condition} would be noncentral and right-hand side would be central, a contradiction.

\end{proof}

\begin{thm} \label{thm:gamma-periodic mag geods when gamma is not central}
Let $\gamma = \exp(V_\gamma + z_\gamma Z) \in H_n$, with $V_\gamma \neq 0$. For each $E > |B|$, there exist two $\gamma$-periodic magnetic geodesic $\sigma(t)$ with energy $E$ and periods $\omega = \pm |V_\gamma|/\sqrt{E^2 - B^2}$. There do not exist any $\gamma$-periodic magnetic geodesics with energy $E \leq |B|$.
\end{thm}

\begin{proof}
By Lemma \ref{lem:non-central elts only translate non-spiraling mag goeds}, we need only consider non-spiraling magnetic geodesics. The energy of any such magnetic geodesic satisfies
\begin{align*}
    E^2 = \sum \frac{u_i^2 + v_i^2}{A_i} + B^2 \geq B^2
\end{align*}
If equality holds, then $\sigma$ is a central 1-parameter subgroup, which is excluded by Lemma \ref{lem:non-central elts only translate non-spiraling mag goeds}. Hence $E > |B|$.

Fix $V_0 \in \mf{v}$ such that its magnitude satisfies $|V_0|^2 + B^2 = E^2$ and its direction is parallel to $V_\gamma$, $V_0 = (B/k)V_\gamma$ for some $k \in \bb{R}_{\neq0}$. Define $\gamma^* = \exp(V_\gamma + kZ)$ and $\sigma^*(t) = \exp(t(V_0 + BZ))$. Then
\begin{align*}
    \gamma^* \sigma^*(t) &= \exp\left( \frac{k}{B}\left(\frac{B}{k}V_\gamma + BZ\right)\right) \exp(t(V_0 + BZ)) \\
    &= \exp\left( \frac{k}{B}\left(V_0 + BZ\right)\right) \exp(t(V_0 + BZ)) \\
    &= \exp\left( \left( t + \frac{k}{B} \right) \left(V_0 + BZ\right) \right) \\
    &= \sigma^*\left( t + \frac{k}{B} \right)
\end{align*}
shows that $\sigma^*$ is a $\gamma^*$-periodic magnetic geodesic of energy $E$ with period $\omega = k/B$. Using the multiplication formula \eqref{eq:multiplication law} and the fact that $Z(\mf{h}_n)$ is one-dimensional, it is straightforward to see that $\gamma$ and $\gamma^*$ are conjugate in $H_n$. Thus, there exists $a \in H_n$ such that $a \gamma^* a^{-1} = \gamma$. Now $\sigma = a \cdot \sigma^*$ is a magnetic geodesic of energy $E$ and
\begin{align*}
    \gamma \cdot \sigma(t) = a \gamma^* a^{-1} \sigma(t) = a \gamma^* \sigma^*(t) = a\sigma^*(t + \omega) = \sigma(t + \omega)
\end{align*}
shows that it is $\gamma$-periodic of period $\omega$. The expression for $\omega$ follows from $\pm k/B=|V_\gamma|/|V_0|,$ and  $|V_0| = \sqrt{E^2 - B^2}$.
\end{proof}



%

Having dealt with the periods of a non-central element of $H_n$, we now consider the case when $\gamma = \exp(z_\gamma Z)$ is central. In this case, there exist  $\gamma$-periodic magnetic geodesics starting at the identity of energy both greater than and less than $|B|$. For a fixed energy $E > |B|$, there will be finitely many distinct periods associated with $\gamma$-periodic magnetic geodesics, while there will be infinitely many distinct periods  when $E < |B|$.

\begin{lem} \label{lem:central elements translate only central 1-param subgroups}
Let $\gamma = \exp(z_\gamma Z)$ for some $z_\gamma \in \bb{R}^*$ and suppose that $\sigma(t)$ is a $\gamma$-periodic magnetic geodesic and a 1-parameter subgroup. Then $\sigma(t) = \exp(tz_0Z)$ for some $z_0 \in \bb{R}^*$. Moreover, for every $E>0$, there exist two $\gamma$-periodic magnetic geodesics of energy $E$, $\sigma(t) = \exp(t(\pm E)Z)$, with period $\omega = z_\gamma / (\pm E)$.
\end{lem}

\begin{proof}
Since $\sigma$ is a 1-parameter subgroup by hypothesis, $\sigma(t) = \exp(tV_0 + BtZ)$. On the one hand $\gamma \sigma(t) = \exp(t V_0 + (Bt + z_\gamma)Z)$ and on the other $\sigma(t + \omega) = \exp((t + \omega)V_0 + B(t + \omega)Z)$. Hence $\omega V_0 = 0$ and since $\omega \neq 0$, we conclude that $V_0 = 0$, showing the first claim.


For each energy $E > 0$, let $z_0 = -B \pm E$ and let $\sigma(t)$ be the magnetic geodesic $\sigma(t) = \exp(t(\pm E)Z)$. Then $\sigma$ is a magnetic geodesic of energy $E$ and
\begin{align*}
    \gamma \sigma(t) = \sigma( (z_\gamma \pm Et) Z) = \exp\left( \pm E \left( \frac{z_\gamma}{\pm E} + t \right) Z \right) = \sigma(t + \omega)
\end{align*}
shows that it is $\gamma$-periodic of period $\omega$.
\end{proof}

Next suppose that $\sigma(t)$ is a spiraling magnetic geodesic, so that the component functions of $\sigma(t)$ have the form \eqref{eq:x_i component of mag geod}-\eqref{eq:z component of mag geod}. Comparing the coefficients of $X_1, \ldots, X_n, Y_1, \ldots, Y_n$ in $\gamma \sigma(t)$ and $\sigma(t + \omega)$ give conditions
\begin{align} \label{eq:resonance condition}
    \sin\left( \frac{z_0}{A_i}(t + \omega) \right) = \sin\left( \frac{z_0}{A_i} t \right) \qquad \cos\left( \frac{z_0}{A_i}(t + \omega) \right) = \cos\left( \frac{z_0}{A_i} t \right)
\end{align}
for each $i = 1, \ldots, n$ such that $u_i^2 + v_i^2 \neq 0$. 

We now specialize to case of the three-dimensional Heisenberg group and obtain a complete description of the spiraling $\gamma$-periodic magnetic geodesics through the identity. Since the left-invariant metric is determined by one parameter, and a magnetic geodesic through the identity is determined by $z_0$ and only one pair of $u_i, v_i$, we write $A = A_1$, $u_0 = u_1$ and $v_0 = v_1$ to ease notation. In general, the analysis will depend on the relative size of $E$ and $B$, and hence breaks up naturally into the three cases $E> |B|$, $E<|B|$ and $E = |B|$. In each case, we first establish the range of permissible integers $\ell$. Next, for each permissible $\ell$, we describe the magnetic geodesics through the identity translated by $\gamma$ along with their respective periods.

In this case, the period $\omega$ and the coordinate $z_0$ must be related by $\omega z_0 = 2\pi A \ell$, where $\ell \in \bb{Z}$. Comparing the central components in $\gamma \sigma(t)$ and $\sigma(t + \omega)$ gives the condition $z(t) + z_\gamma = z(t + \omega)$. That is,
\begin{align*}
    & \left(z_0 + B + \frac{u_0^2 + v_0^2}{2A z_0} \right) t - \frac{u_0^2 + v_0^2}{2 z_0^2} \sin \left(\frac{z_0 t}{A}\right) + z_\gamma \\ & \qquad = \left(z_0 + B + \frac{u_0^2 + v_0^2}{2A z_0} \right) (t + \omega) - \frac{u_0^2 + v_0^2}{2 z_0^2} \sin \left(\frac{z_0 }{A} (t + \omega) \right).
\end{align*}
This simplifies to
\begin{align} \label{eq:central periodic condition 1}
    z_\gamma = \left(z_0 + B + \frac{u_0^2 + v_0^2}{2A z_0} \right) \omega,
\end{align}
and using  \eqref{eq:energy of spiraling mag geod} to eliminate the fraction and $\omega z_0 = 2\pi A \ell$ to eliminate $\omega$ this can be written as
\begin{align} \label{eq:central periodic condition 2}
    z_\gamma &= \left( z_0 + B + \frac{1}{2z_0}(E^2 - (z_0 + B)^2 ) \right) \frac{2\pi A \ell}{z_0}.
\end{align}
If $E = |B|$, then the above simplifies to $z_\gamma = \pi A \ell$. If $E \neq |B|$, then after clearing denominators and solving for $z_0$, we obtain the expression
\begin{align} \label{eq:central periodic condition 3}
    z_0^2 = \frac{E^2 - B^2}{\frac{z_\gamma}{\pi A \ell} - 1}.
\end{align}

\begin{lem} \label{lem:effect of energy constraint on range of ells}
Let $\gamma = \exp(z_\gamma Z)$ be a central element of the Heisenberg group. For each nonzero energy level, the range of admissible integers $\ell$ and the corresponding choices of $z_0$ for which there exists a $\gamma$-periodic magnetic geodesic through the identity are given by the following table.
\begin{center}
{\renewcommand{\arraystretch}{2}
\begin{tabular}{l|l|l|l}
 &  & \multicolumn{1}{|c}{$\ell$} & \multicolumn{1}{|c}{$z_0$} \\ \hline
(1a)    & $E > |B|$       & $1 < \frac{2E}{E + B} < \frac{z_\gamma}{\pi A \ell}$    & $-\sqrt{\frac{E^2 - B^2}{\frac{z_\gamma}{\pi A \ell} - 1}}$ \\ \hline
(1b)    & $E > |B|$       & $1 < \frac{2E}{E - B} < \frac{z_\gamma}{\pi A \ell}$    & $+\sqrt{\frac{E^2 - B^2}{\frac{z_\gamma}{\pi A \ell} - 1}}$ \\ \hline
(2a)     & $0 < E < B$   & $\frac{2E}{E-|B|} < \frac{z_\gamma}{\pi A \ell} < \frac{2E}{E+|B|} < 1$ & $-\sqrt{\frac{E^2 - B^2}{\frac{z_\gamma}{\pi A \ell} - 1}}$ \\ \hline
(2b)     & $B < E < 0$   & $\frac{2E}{E-|B|} < \frac{z_\gamma}{\pi A \ell} < \frac{2E}{E+|B|} < 1$ & $+\sqrt{\frac{E^2 - B^2}{\frac{z_\gamma}{\pi A \ell} - 1}}$ \\ \hline
(3a)     & $E = B$       & $\ell = \frac{z_\gamma}{\pi A}$     &   $-2B < z_0 < 0$        \\ \hline
(3b)     & $E = -B$       & $\ell = \frac{z_\gamma}{\pi A}$     &  $0 < z_0 < -2B$         
\end{tabular}}
\end{center}
In all cases, the associated period is $\omega = 2\pi A\ell / z_0$ and one can choose any $u_0$ and $v_0$ such that $u_0^2 + v_0^2 = A(E^2 - (z_0 + B)^2)$.
\end{lem}

\begin{proof}
The condition $(z_0 + B)^2 < E^2$ is equivalent to
\begin{align} \label{eq:basic energy constraint}
    -E-B < \pm \sqrt{\frac{E^2 - B^2}{\frac{z_\gamma}{\pi A \ell} - 1}} < E-B.
\end{align}

In case (1), $-E - B < 0$ and $E - B > 0$, so this leads to the two inequalities
\begin{align*}
    -E-B < - \sqrt{\frac{E^2 - B^2}{\frac{z_\gamma}{\pi A \ell} - 1}} < 0, \hspace{2cm} 0 < \sqrt{\frac{E^2 - B^2}{\frac{z_\gamma}{\pi A \ell} - 1}} < E - B.
\end{align*}
After squaring both inequalities and isolating $z_\gamma/(\pi A \ell)$, these become
\begin{align*}
    1 < \frac{2E}{E+B} < \frac{z_\gamma}{\pi A \ell} \hspace{2cm} 1 < \frac{2E}{E-B} < \frac{z_\gamma}{\pi A \ell}
\end{align*}
yielding cases (1a) and (1b), respectively. Notice that one of these ranges for $\ell$ is a subset of the other. We keep them separate as they affect the choice of sign for $z_0$.

In case (2), either $-B-E < -B+E < 0$ if $B > 0$, or $0 < -B-E < -B+E < 0$ if $B < 0$. A similar computation as above leads to the inequalities
\begin{align*}
    \frac{2E}{E-B} < \frac{z_\gamma}{\pi A \ell} < \frac{2E}{E + B} < 1, \hspace{1cm} \frac{2E}{E+B} < \frac{z_\gamma}{\pi A \ell} < \frac{2E}{E-B} < 1.
\end{align*}
Both of these ranges can be expressed simultaneously in terms of $|B|$ as in the Lemma statement. However, in case (2a), when $B> 0$, $z_0$ is chosen according to the negative branch, and vice versa in case (2b).

For case (3), it was noted above \eqref{eq:central periodic condition 3} that if $E = |B|$, then $z_\gamma = \pi A \ell$. Choose $z_0$ so that $(z_0 + B)^2 < E^2$. When $B>0$, this inequality is the same as $-2B < z_0 < 0$. Setting $\omega = 2\pi A \ell / z_0 = 2z_\gamma / z_0$, it is straightforward to check that $\frac{z_0}{A}(t + \omega) = \frac{z_0}{A}t$ and that \eqref{eq:central periodic condition 1} holds. The case when $B<0$ is handled similarly.
\end{proof}

\begin{rmk}
In case (2), the condition that $E< |B|$ ensures that $E^2 - B^2 < 0$, while the conditions on $\ell$ ensure that $z_\gamma/(\pi A \ell) - 1 < 0$. Hence the expression under the radical in $z_0$ will be positive.
\end{rmk}

\begin{rmk}
In every case, for each admissible $z_0$ there is a 1-parameter family of $\gamma$-periodic magnetic geodesics.
\end{rmk}


\begin{rmk} \label{rmk:sometimes critical case is empty}
The cases where $E = |B|$ are to be interpreted as follows. When $z_\gamma$ and $A$ are such that $z_\gamma / \pi A \in \bb{Z}$, then there exist $\gamma$-periodic magnetic geodesics with energy $E$ and $z_0$ as described in the table. Otherwise, the collection of such magnetic geodesics is empty.
\end{rmk}

\subsection{Lengths of Closed Magnetic Geodesics}

We are now in a position to compute the lengths of closed magnetic geodesics on $\Gamma \backslash H$ in the free homotopy class  of $\gamma \in \Gamma$. If $\Gamma < H$ is a cocompact discrete subgroup, and $\gamma \in \Gamma$, then the length of the corresponding closed magnetic geodesic on the compact quotient $\Gamma \backslash H$ will be
\begin{align} \label{eq:length of gamma-periodic geodesic}
    \int_0^{|\omega|} |\sigma'(t)| dt = E|\omega|.
\end{align}
Previous results results concerning the lengths of closed geodesics in the Riemannian case include \cite{gordon1986spectrum}, \cite{eberlein1994geometry}, \cite{gornetmast2000lengthspectrum}, \cite{gornetmast2003}. Unlike the Riemannian case, magnetic geodesics cannot be reparamterized to have a different energy. So it is more natural to consider the collection of lengths of closed geodesics of a fixed energy. Let $L(\gamma; E)$ denote the set of distinct lengths of closed magnetic geodesics in the free homotopy class of $\gamma$.

\begin{thm} \label{thm:periods of central elt in 3D Heisenberg}
Let $\Gamma < H$ be a cocompact discrete subgroup of the Heisenberg group $H$ and let $\gamma = \exp(V_\gamma + z_\gamma Z) \in \Gamma$. 
\begin{itemize}
    \item If $\gamma = e$ is the identity ($V_\gamma = 0$ and $z_\gamma = 0$), then 
    \begin{align} \label{eq:lengths in trivial free homotopy class}
        \displaystyle L(e;E) = \begin{cases} \emptyset & \text{if } E \geq |B| \\ \left\{\frac{2\pi A}{\sqrt{\frac{B^2}{E^2} - 1}}\right\} & \text{if } 0 < E < |B| \end{cases}
    \end{align}
    \item If $\gamma$ is not central ($V_\gamma \neq 0$) then 
    \begin{align} \label{eq:lengths in noncentral free homotopy classes}
        \displaystyle L(\gamma;E) = \begin{cases} \emptyset & \text{if } 0 < E \leq |B| \\ \left\{ \frac{|V_\gamma|}{\sqrt{1 - \frac{B^2}{E^2}}} \right\} & \text{if } E > |B| \end{cases}
    \end{align}
    \item If $\gamma$ is central ($V_\gamma = 0$ and $z_\gamma \neq 0$) then
    \begin{align} \label{eq:lengths in central free homotopy classes}
        & L(\gamma; E) = \\ & \begin{cases} \left\{ \frac{\sqrt{4 \pi A \ell (z_\gamma - \pi A \ell)}}{\sqrt{1 - \frac{B^2}{E^2}}} : \ell \in \bb{Z},  \frac{2E}{E + |B|} < \frac{z_\gamma}{\pi A \ell} \right\} \cup \left\{ |z_\gamma| \right\} & E > |B| \\  \left\{ \frac{\sqrt{4 \pi A \ell (\pi A \ell - z_\gamma)}}{\sqrt{\frac{B^2}{E^2} - 1}} : \ell \in \bb{Z}, \frac{2E}{E-|B|} < \frac{z_\gamma}{\pi A \ell} < \frac{2E}{E+|B|} \right\} \cup \left\{ |z_\gamma| \right\} & 0 < E < |B| \\  \left\{ \frac{2E |z_\gamma|}{|z_0|} : z_0 \in \bb{R}, (z_0 + B)^2 < E^2 \right\} \cup \left\{ |z_\gamma| \right\} & E = |B| \end{cases} \nonumber
    \end{align}
\end{itemize}
\end{thm}

\begin{proof}
The case when $\gamma = e$ follows from Theorem \ref{thm:closed contractible magnetic geodesics}. The lengths of closed magnetic geodesics obtained in that theorem is
\begin{align*}
    E|\omega| = E \left| \frac{2\pi A}{-\sgn(B)\sqrt{B^2 - E^2}} \right| = \frac{2\pi AE}{\sqrt{B^2 - E^2}}
\end{align*}
The case when $\gamma = \exp(V_\gamma + z_\gamma Z)$ is not central follows from Theorem \ref{thm:gamma-periodic mag geods when gamma is not central}. The length of closed magnetic geodesics obtain in that theorem is
\begin{align*}
    E|\omega| = E \left| \frac{|V_\gamma|}{\sqrt{E^2 - B^2}} \right| = \frac{E|V_\gamma|}{\sqrt{E^2 - B^2}}.
\end{align*}
The case when $\gamma$ is central follows from Lemma \ref{lem:central elements translate only central 1-param subgroups} and Lemma \ref{lem:effect of energy constraint on range of ells}. In the former case, which applies to every energy, the length of the closed magnetic geodesic is
\begin{align*}
    E|\omega| = E \left| \frac{z_\gamma}{\pm E} \right| = |z_\gamma|.
\end{align*}
In the latter case, when $E > |B|$ the lengths are
\begin{align*}
    E|\omega| =  E \left| \frac{2\pi A \ell}{z_0} \right| = 2\pi AE \ell \left| \pm \sqrt{\frac{\frac{z_\gamma}{\pi A \ell} - 1}{E^2 - B^2}} \right| = \frac{2E \sqrt{\pi A \ell (z_\gamma - \pi A \ell)}}{\sqrt{E^2 - B^2}}
\end{align*}
and when $E < |B|$ the lengths are
\begin{align*}
    E|\omega| =  E \left| \frac{2\pi A \ell}{z_0} \right| = 2\pi AE \ell \left| \pm \sqrt{\frac{1 - \frac{z_\gamma}{\pi A \ell}}{B^2 - E^2}} \right| = \frac{2E \sqrt{\pi A \ell (\pi A \ell - z_\gamma)}}{\sqrt{B^2 - E^2}}.
\end{align*}
The lengths when $E = |B|$ depend not on $\ell$ (which must be $\ell = z_\gamma / (\pi A)$) but instead on $z_0$ and are given by
\begin{align*}
    E|\omega| = E\left| \frac{2\pi A \ell}{z_0} \right| = E\left| \frac{2 z_\gamma }{z_0} \right|.
\end{align*}
\end{proof}

\begin{rmk}
As $E \to \infty$ or $B \to 0$, the denominator $\sqrt{1 - B^2/E^2} \to 1$. Roughly speaking, the cases $E \leq |B|$ will be eliminated, and the collection of lengths in the case $E > |B|$ will approach the length spectrum in the Riemannian case, which was computed in \cite{gornetmast2000lengthspectrum}. This reflects the following physical intuition: when the magnetic field is very weak charged particles will behave more like they would in the absence of any forces, and when a particle is very energetic the magnetic field will have less of an effect on its trajectory.
\end{rmk}

\begin{rmk}
The dynamics of the magnetic flow on the various energy levels splits roughly into three regimes:
\begin{itemize}
    \item For fixed energy levels $E > |B|$, there exist closed magnetic geodesics in every free homotopy class and the set of their lengths is finite.
    \item For fixed energy levels $E<|B|$, there exist free homotopy classes without any closed magnetic geodesics, and in the case that there are closed magnetic geodesics, the set of their lengths is countably infinite. This reflects the paradigm that the dynamics on high energy levels will resemble that of the underlying geodesic flow.
    \item Finally, when $E = |B|$,  $\gamma$ is central, and $z_\gamma \in \pi A \bb{Z}$ (i.e. the set of lengths is nonempty),  then the infinite set of lengths is not discrete.
\end{itemize}
\end{rmk}

The following three lemmas address bounds on the collection of lengths of closed magnetic geodesics in a given central free homotopy class.

\begin{lem} \label{lem:upper bound for supercritical lengths}
Consider the case $|B|<E$ in \eqref{eq:lengths in central free homotopy classes}. The set 
\begin{align*}
    \left\{ \frac{\sqrt{4 \pi A \ell (z_\gamma - \pi A \ell)}}{\sqrt{1 - \frac{B^2}{E^2}}} : \ell \in \bb{Z},  \frac{2E}{E + |B|} < \frac{z_\gamma}{\pi A \ell} \right\}
\end{align*}
is bounded above by $|z_\gamma|/\sqrt{1-B^2/E^2},$ which is larger than $|z_\gamma|.$ The example below shows that this upper bound is the best possible.
\end{lem}

\begin{proof}
Without loss of generality, we assume $z_\gamma>0.$ The condition on $\ell$ implies $0<\ell<\frac{z_\gamma}{2\pi A}\left(1+\frac{|B|}{E}\right). $  We define
\begin{align*}
    \lambda(\ell)=\frac{4 \pi A \ell (z_\gamma - \pi A \ell)}{\left(1 - \frac{B^2}{E^2}\right)}.
\end{align*}
The parabola $\lambda(\ell)$ opens downward and has zeroes at $\ell=0$ and $\ell=z_\gamma/\pi A,$ hence achieves a maximum of $z^2_\gamma/\left(1-B^2/E^2\right)$ at $\ell=z_\gamma/2\pi A.$ See the example below for values of $A,B,z_\gamma$ such that this maximum is achieved. The result follows.
\end{proof}

\begin{ex} \label{ex:maximal length need not be straight line}
Consider the particular example where $A = 1$, $B = 1$ and $E=2$. Choose the central element $\gamma = \exp(20\pi Z)$ so that $z_\gamma = 20\pi$. In this case, for each $\ell$ such that $0 < \ell < 15$, there is a closed magnetic geodesic with length given by \eqref{eq:lengths in central free homotopy classes}. In particular, when $\ell = 10$ the corresponding length is $(2/\sqrt{3})20 \pi > z_\gamma$.
\end{ex}

\begin{rmk} \label{rmk:maximal mag length spec not well behaved}
In the setting of Riemannian two-step nilmanifolds, the maximal length of a closed magnetic geodesic in a central free homotopy class is the length of the central geodesic. In fact, the maximal length spectrum determines the length spectrum for central free homotopy classes (see Proposition 5.15 of \cite{eberlein1994geometry}). Example \ref{ex:maximal length need not be straight line} shows that this is no longer true in the magnetic setting.
\end{rmk}

\begin{lem}
Consider the case $|B|>E$ in \eqref{eq:lengths in central free homotopy classes}. The set
\begin{align}
    \left\{ \frac{\sqrt{4 \pi A \ell (\pi A \ell - z_\gamma)}}{\sqrt{\frac{B^2}{E^2} - 1}} : \ell \in \bb{Z}, \frac{2E}{E-|B|} < \frac{z_\gamma}{\pi A \ell} < \frac{2E}{E+|B|} \right\} 
\end{align}
is bounded below by $|z_\gamma|$. 
\end{lem}

\begin{proof}
Without loss of generality, we assume $z_\gamma>0.$
We define  $$\lambda(\ell)=\frac{4 \pi A \ell (\pi A \ell-z_\gamma)}{\left(\frac{B^2}{E^2}-1\right)}.$$
The parabola $\lambda(\ell)$ opens upward and has zeroes at $\ell=0$ and $\ell=z_\gamma/\pi A.$
The condition on $\ell$ implies $\ell>\frac{z_\gamma}{2\pi A}\left(1+\frac{|B|}{E}\right)>\frac{z_\gamma}{\pi A}$ or $\ell<\frac{z_\gamma}{2\pi A}\left(1-\frac{|B|}{E}\right)<0.$  A lower bound of the set is thus provided by the minimum of 
$\sqrt{\lambda\left(\frac{z_\gamma}{2\pi A}\left(1+\frac{|B|}{E}\right)\right)}$ 
and $\sqrt{\lambda\left(\frac{z_\gamma}{2\pi A}\left(1-\frac{|B|}{E}\right)\right)}.$
However, both of these evaluate to $z_\gamma,$ and the result follows.
\end{proof}

\begin{lem}
Consider the case $|B|=E$ in \eqref{eq:lengths in central free homotopy classes}. If the set
\begin{align*}
\left\{ \frac{2E |z_\gamma|}{|z_0|} : z_0 \in \bb{R}, (z_0 + B)^2 < E^2 \right\}    
\end{align*}
is nonempty (see Remark \ref{rmk:sometimes critical case is empty}), then it is unbounded above and has an infimum of $|z_\gamma|.$
\end{lem}

\begin{proof}
If $B > 0$, then $z_0$ can be chosen in the interval $-2B < z_0 < 0$. As $z_0 \to 0^-$, the length diverges to infinity, and as $z_0 \to (-2B)^+$ the lengths converge to $|z_\gamma|$. The case when $B<0$ is analogous.
\end{proof}


\subsection{Density of Closed Magnetic Geodesics}


Given a Riemannian manifold $M$, define $S^E M=\{V \in TM: |V| = E\}$ and let $S_\gamma^E M$ denote the tangent sphere of radius $E$ at the point $\gamma$. Given a vector $V \in TM$, let $\sigma_V$ denote the magnetic geodesic such that $\sigma_V'(0) = V$. We are interested in the size of the set of vectors that determine periodic magnetic geodesics. In the Riemannian case, this set is scale invariant. That is, if $V$ determines a periodic geodesic, then so does $cV$ for any $c \neq 0$. So it is natural in this case to restrict attention to unit vectors. However, this property does not hold  for magnetic geodesics. Therefore, in the following definition we include a dependence on the energy of the vectors.
\begin{align} \label{eq:defn of set of periodic vectors}
    \Per^E(M) := \{ V \in S^E M \ : \ \sigma_V \text{ is periodic} \} \subset S^EM.
\end{align}
In the context of Riemannian two-step nilmanifolds, the density of this set was first investigated in \cite{eberlein1994geometry}, and subsequently in \cite{mast1994closedgeods}, \cite{mast1997resonance}, \cite{leepark1996density}, \cite{demeyer2001compactnilmanifolds}, \cite{decoste2008chevalleyratstructs}. The following result shows that for magnetic flows on the Heisenberg group, density persists for sufficiently high energy.

\begin{thm} \label{thm:density of closed mag geods}
For each $E > |B|$, $\Per^E(\Gamma \backslash H)$ is dense in $S^E(\Gamma \backslash H)$.
\end{thm}

\begin{proof}

We begin with a series of reductions. First, it suffices to show that the set of $V \in S^E(H)$ such that $\sigma_V$ is $\gamma$-periodic for some $\gamma \in \Gamma$ is dense in $S^E(H)$. For any $V \in S^E(\Gamma \backslash H)$, let $W \in \pi^{-1}(V)$ and let $\{ W_i \} \subset S^E(H)$ be such that $\sigma_{W_i}$ is $\gamma_i$-periodic for some $\gamma_i \in \Gamma$ and $W_i \to W$. Then $\{ V_i = \pi(W_i) \} \subset S^E(\Gamma \backslash H)$ is a sequence of tangent vectors such that $\sigma_{V_i}$ is periodic and $V_i \to V$.

Next, we claim that it suffices to show that the set of $W \in S^E_e H$ such that $\sigma_W$ is periodic for some $\gamma \in Z(\Gamma)$ is dense in $S^E_e H$. For if $\sigma_W$ is such a magnetic geodesic and $\phi \in H$ is any element, then $\phi \cdot \sigma_W(t)$ is a $(\phi \gamma \phi^{-1})$-periodic magnetic geodesic satisfying $\phi \cdot \sigma_V(0) = \phi$ and $(\phi \cdot \sigma_V)'(0) = L_{\phi *}(V)$. Because $\gamma$ is central, $\phi \gamma \phi^{-1} = \gamma$ and $\phi \cdot \sigma_W$ is a $\gamma$-periodic magnetic geodesic. Because $L_{\phi *} S_e^E(H) \to S_\phi^E(H)$ is a diffeomorphism, this proves the claim.

Lastly, we claim that it suffices to show that set $z_0 \in [-B-E, -B+E]$ chosen according to cases (1a) and (1b) in Lemma \ref{lem:effect of energy constraint on range of ells} (for some choice of $\gamma \in Z(\Gamma)$) is dense in $[-B-E, -B+E]$. As noted in Lemma \ref{lem:effect of energy constraint on range of ells}, for any such $z_0$ there is a one parameter family of $\gamma$-periodic magnetic geodesics given by any choice of $u_0, v_0$ such that $u_0^2 + v_0^2 = A(E^2 - (z_0 + B)^2)$. Hence if the resulting $z_0$ are dense in $[-B-E, -B+E]$, then there is a dense set of latitudes in the ellipsoid $E^2 = ((u_0^2 + v_0^2)/A) + (z_0 + B)^2 \subset \bb{R}^3$ such that those vectors yield $\gamma$-periodic magnetic geodesics for some $\gamma \in \Gamma$. The initial conditions $(u_0, v_0, z_0) \in \bb{R}^3$ determine the magnetic geodeisc $\sigma_V$ where $V = (u_0/A)X + (v_0/A)Y + (z_0 + B)Z$, showing that the set of $V \in S^E_e H$ tangent to $\gamma$-periodic magnetic geodesics ($\gamma \in \Gamma$) is dense in $S^E_e H$.

By Proposition 5.4 of \cite{eberlein1994geometry}, $\Gamma \cap Z(H) = Z(\Gamma)$ is a lattice in $Z(H)$. Hence there exists  $\bar{z} \in \bb{R}^*$ such that $\Gamma \cap Z(H) = \{ \exp(h\bar{z}Z) \ : \ h \in \bb{Z} \}$. By replacing $\bar{z}$ with $-\bar{z}$, if necessary, we can assume that $\bar{z} > 0$. Consider the set of numbers
\begin{align*}
    \left\{ \frac{h}{\ell} \ : \ h, \ell \in \bb{Z}^+ \text{ and } \left( \frac{2\pi AE}{\bar{z}(E + B)} \right) \ell < h \right\}.
\end{align*}
This set is dense in the interval $(2\pi AE/(\bar{z}(E + B)), \infty)$. Via a sequence of continuous mappings of $\bb{R}$, each of which preserves density, \begin{align*}
    \left\{ - \sqrt{\frac{E^2 - B^2}{\frac{h\bar{z}}{\pi A \ell} - 1}} \ : \ h, \ell \in \bb{Z}^+ \text{ and } \left( \frac{2 E}{E + B} \right) \ell < h \right\}
\end{align*}
is dense in the interval $(-E-B,0)$. These are preciesly the values for $z_0$ appearing in case (1a) of Lemma \ref{lem:effect of energy constraint on range of ells}. Starting instead with the set
\begin{align*}
    \left\{ \frac{h}{\ell} \ : \ h, \ell \in \bb{Z}^+ \text{ and } \left( \frac{2\pi AE}{\bar{z}(E - B)} \right) \ell < h \right\}
\end{align*}
and using a parallel sequence of transformations shows that
\begin{align*}
    \left\{ \sqrt{\frac{E^2 - B^2}{\frac{h\bar{z}}{\pi A \ell} - 1}} \ : \ h, \ell \in \bb{Z}^+ \text{ and } \left( \frac{2 E}{E - B} \right) \ell < h \right\}
\end{align*}
is dense in $(0, E - B)$. These numbers are the $z_0$ appearing in case (1b) of Lemma \ref{lem:effect of energy constraint on range of ells}. This shows the density of permissible $z_0$ in the interval $[-E-B, E-B]$ and hence the theorem.

\end{proof}

\subsection{Rigidity and the Marked Magnetic Length Spectrum}

\label{sec:rigidity}

We begin by recalling the notion of marked length spectrum for a compact Riemannian manifold $M$. For each nontrivial free homotopy class $\mc{C}$, there exists at least one smoothly closed Riemannian geodesic. Let $L(\mc{C})$ denote the collection of all lengths of smooth closed geodesics that belong to $\mc{C}$. Recall that free homotopy classes of closed curves on $M$ are in bijection with conjugacy classes of $\pi_1(M)$. If $\bar{M}$ is another compact Riemannian manifold and $\phi : \pi_1(M) \to \pi_1(\bar{M})$ is an isomorphism, then $\phi$ maps conjugacy classes of $\pi_1(M)$ bijectively onto conjugacy classes of $\pi_1(\bar{M})$. Hence $\phi$ induces a bijection $\phi_*$ of the set of free homotopy classes of closed curves on $M$ onto the set of free homotopy of classes of closed curves on $\bar{M}$. Two compact Riemannian manifolds $M$ and $\bar{M}$ are said to have the  \emph{same marked length spectrum} if there exists an isomorphism $\phi : \pi_1(M) \to \pi_1(\bar{M})$ such that $L(\phi_* \mc{C}) = L(\mc{C})$ for all nontrivial free homotopy classes of closed curves on $M$. Specializing to the case at hand, let $G_1$ and $G_2$ be two simply connected 2-step nilpotent Lie groups and $\Gamma_1 < G_1$ and $\Gamma_2 < G_2$ cocompact discrete subgroups. Then $\pi_1(\Gamma_i \backslash G_i) \simeq \Gamma_i$. With these identifications, we say the  nilmanifolds $\Gamma_1 \backslash G_1$ and $\Gamma_2 \backslash G_2$ have the same marked length spectrum if there is an isomorphism $\phi: \Gamma_1 \to \Gamma_2$ such that $L(\phi_* \mc{C}) = L(\mc{C})$ for all nontrivial free homotopy classes of closed curves on $\Gamma_1 \backslash G_2$.  See \cite{eberlein1994geometry} and \cite{gornetmast2004} for previous results on marked length spectrum rigidity of Riemannian two-step nilmanifolds.

While the above definition could be used in the context of magnetic flows on nilmanifolds, it seems more natural to modify it in light of the dependence of the dynamics on the relative magnitudes of $E$ and $|B|$. For a fixed homotopy class $\mc{C}$, the collection lengths of closed magnetic geodesics of any energy could be an infinite open interval. Therefore, let $L(\mc{C}; E)$ denote the collection of all lengths of smoothly closed magnetic geodesics that belong to $\mc{C}$ and have energy $E$. By Theorem \ref{thm:periods of central elt in 3D Heisenberg}, $L$ is not well-defined for $E \leq |B|$. In order to avoid this, we only define $L$ for $E > |B|$.

\begin{defn}
Let $H$ be the simply connected three-dimensional Heisenberg group. Let $g_1$ and $g_2$ be two left-invariant Riemannian metrics on $H$ with parameters $A_1$ and $A_2$. Let $\Omega_1$ and $\Omega_2$ be two left-invariant magnetic forms on $H$ with parameters $B_1$ and $B_2$. Let $\Gamma_1, \Gamma_2 < H$ be two cocompact discrete subgroups, and $\phi: \Gamma_1 \to \Gamma_2$ an isomorphism. The nilmanifolds $\Gamma_1 \backslash H$ and $\Gamma_2 \backslash H$ with corresponding magnetic structures are said to have the  \emph{same marked magnetic length spectrum} if $L(\phi_* \mc{C}; E_2) = L(\mc{C}; E_1)$ for some $E_1 > |B_1|$ and some $E_2 > |B_2|$ for each nontrivial free homotopy class $\mc{C}$.
\end{defn}

Even though the magnetic flow is a perturbation away from the underlying geodesic flow, it reflects enough of the underlying Riemannian geometry to exhibit a degree of geometric rigidity.

\begin{thm} \label{thm:MLS rigidity}
Let $H$ be the simply connected, three-dimensional Heisenberg group endowed with left-invariant Riemannian metric $g$ and left-invariant magnetic form $\Omega$, with corresponding parameters $A$ and $B$ respectively. Let $\Gamma_1, \Gamma_2 < H$ be two cocompact lattices. Suppose that for some $E > |B|$, the two manifolds $\Gamma_1 \backslash H$ and $\Gamma_2 \backslash H$ have the same marked magnetic length spectrum at energy $E$. Then $\Gamma_1 \backslash H$ and $\Gamma_2\backslash H$ are isometric.
\end{thm}

The proof of Theorem \ref{thm:MLS rigidity} is similar to the proof of Theorem 5.20 in \cite{eberlein1994geometry}, with one notable exception. The latter uses the maximal marked length spectrum, i.e. only the length longest closed geodesic in each free homotopy class. For Riemannian geodesics in central free homotopy classes (on two-step nilpotent Lie groups), this is always length of the one-paramter subgroup. Example \ref{ex:maximal length need not be straight line} and Remark \ref{rmk:maximal mag length spec not well behaved} show that the maximal magnetic marked length spectrum is not so well behaved. To circumvent this, we consider all the lengths of closed magnetic geodesics in central free homotopy classes. This argument is given in the following Lemma.

\begin{lem} \label{lem:rigidity of central lattice}
Under the same hypotheses as Theorem \ref{thm:MLS rigidity}, let $\exp(\bar{z}_1 Z)$ and $\exp(\bar{z}_2 Z)$ be generators for the central lattices $\Gamma_1 \cap H$ and $\Gamma_2 \cap H$, respectively. Then $|\bar{z}_1| = |\bar{z}_2|$.
\end{lem}

\begin{proof}
First, we claim that
\begin{align} \label{eq:normalized lengths of closed mag geods in central free homotopy classes}
    \sup_{h \in \bb{Z}} \left\{ \frac{\max(L([\exp(h\bar{z}_1 Z)]_1; E))}{|h|} \right\} = \sup_{h \in \bb{Z}} \left\{ \frac{\max(L([\exp(h\bar{z}_2 Z)]_2; E))}{|h|} \right\}
\end{align}
where $[\gamma]_i$ denotes the free homotopy class of closed curves on $\Gamma_i \backslash H$ determined by $\gamma \in \Gamma_i$. Let $\phi: \Gamma_1 \to \Gamma_2$ be an isomorphism. Since $\phi$ is an isomorphism of $Z(\Gamma_1)$ onto $Z(\Gamma_2)$, $\phi(\exp(h\bar{z}_1 Z)) = \exp(\pm h \bar{z}_2 Z)$, and so $\phi_*[\exp(h\bar{z}_1 Z)]_1 = [\exp(\pm h \bar{z}_2 Z)]_2$. By hypothesis, the sets of lengths of closed magnetic geodesics in these two classes are equal. Moreover, the positive integer $|h|$ is the same for both free homotopy classes. Hence the sets over which the supremums are taken are equal.

Next we evaluate the supremums in \eqref{eq:normalized lengths of closed mag geods in central free homotopy classes}. By Lemma \ref{lem:upper bound for supercritical lengths}, the set of lengths of smoothly closed magnetic geodesics in the free homotopy class determined by an element of the form $\exp(h\bar{z}_i Z)$ is bounded above by $|h\bar{z}_i|/\sqrt{1 - B^2 / E^2}$. After dividing all the lengths in each set by $|h|$, respectively, we obtain a uniform upper bound,
\begin{align} \label{eq:uniform upper bound for normalized lengths}
    \sup_{h \in \bb{Z}} \left\{ \frac{\sqrt{4 \pi A \ell (h\bar{z}_i - \pi A \ell)}}{|h| \sqrt{1 - \frac{B^2}{E^2}}} : \ell \in \bb{Z}, \ \frac{2E}{E + |B|} < \frac{h\bar{z}_i}{\pi A \ell} \right\} \leq \frac{|\bar{z}_i|}{\sqrt{1 - \frac{B^2}{E^2}}}.
\end{align}
We now claim that the inequality in \eqref{eq:uniform upper bound for normalized lengths} is actually an equality. If the quantity $\bar{z}_i/(2\pi A) \in \bb{Q}$, then for $h$ large enough $\ell = (h\bar{z}_i)/(2\pi A)$ will be an allowable integer value for $\ell$, and $\max_{\ell}(\sqrt{4 \pi A \ell (h\bar{z}_i - \pi A \ell)}) = |h\bar{z}_i|$. If the quantity $\bar{z}_i/(2\pi A) \notin \bb{Q}$, then the numbers $(h\bar{z}_i)/(2\pi A)$ will will come arbitrarily close to an integer. In either case, the supremum is $|\bar{z}_i|/\sqrt{1 - B^2/E^2}$. The lemma now follows.
\end{proof}

We now proceed with the proof of Theorem \ref{thm:MLS rigidity}.

\begin{proof}
First, extend the marking $\phi : \Gamma_1 \to \Gamma_2$ to an automorphism $\phi: H \to H$. Because $\phi_*$ is a Lie algebra automorphism of $\mf{h} = \mf{v} \oplus \mf{z}$, we can decompose it as $\phi_* = R_1 + R_2 + S$, where $R_1 : \mf{v} \to \mf{v}$, $R_2 : \mf{v} \to \mf{z}$, and $S : \mf{z} \to \mf{z}$ are linear maps. Using Lemma \ref{lem:rigidity of central lattice},
\begin{align*}
    |S(\bar{z}_1Z)| = |\pm \bar{z}_2 Z| = |\bar{z}_2| = |\bar{z}_1| = |\bar{z}_1 Z|
\end{align*}
shows that $S$ is an isometry of $\mf{z}$. Let $\pi_\mf{v}:\mf{h} \to \mf{v}$ denote the projection. For any $V \in \pi_\mf{v} \log \Gamma_1$, there is some $\xi \in \Gamma_1$ such that $\xi = V + Z$ and $Z \in \mf{z}$. By hypothesis, $\exp(\xi)$ and
\begin{align*}
    \phi(\exp(V+Z)) = \phi(\exp(\xi)) = \exp(\phi_* \xi) = \exp( R_1(V) + R_2(V) + S(Z))
\end{align*}
have the same lengths of closed magnetic geodesics. By \eqref{eq:lengths in noncentral free homotopy classes}, we have
\begin{align*}
    \frac{|V|}{\sqrt{1 - \frac{B^2}{E^2}}} = \frac{|R_1(V)|}{\sqrt{1 - \frac{B^2}{E^2}}}
\end{align*}
and we conclude that $R_1$ is an isometry of $\mf{v}$. It is straightforward to check that $R_1 + S : \mf{h} \to \mf{h}$ is an isometric Lie algebra isomorphism. Let $\phi_1 : H \to H$ be the Lie group isomorphism such that $(\phi_1)_* = R_1 + S$. Define $T : \mf{h} \to \mf{h}$ by $T(V + Z) = V + Z + (S^{-1} \circ R_2)(V)$. Once can verify directly that $T$ is an inner automorphism of $\mf{h}$ and $(\phi_1)_* \circ T = \phi_*$. Let $\phi_2$ be the inner automorphism of $H$ such that $(\phi_2)_* = T$.

Now we have that $\phi = \phi_1 \circ \phi_2$, where $\phi_1$ is an isometric automorphism of $H$ and $\phi_2$ is an in inner automorphism of $H$. Because $\phi_2$ is an inner automorphism, $\Gamma_1 \backslash H$ is isometric to $\phi_2(\Gamma_1) \backslash H$ (via a left-translation), and $\phi_2(\Gamma_1) \backslash H$ is isometric to $\phi_1(\phi_2(\Gamma_1)) \backslash H = \phi(\Gamma_1) \backslash H = \Gamma_2 \backslash H$.
\end{proof}

\begin{rmk}
For $E < |B|$, the set of lengths in any noncentral free homotopy class is empty. Hence magnetic length spectrum does not determine 
\end{rmk}

\begin{rmk}
The isometry $\Gamma_1 \backslash H \to \Gamma_2 \backslash H$ preserves the magnetic form $\Omega = d(B \zeta)$ up to sign. Since the isometry is realized by $\phi_1 \circ L_x$ for some $x \in H$,
\begin{align*}
    (\phi_1 \circ L_x)^* \zeta = L_x^* (\phi_1^* \zeta) = L_x^*(\pm \zeta) = \pm \zeta.
\end{align*}
\end{rmk}

\section{Heisenberg Type Manifolds} \label{sec:HT manifolds}

Heisenberg type manifolds are Riemannian manifolds that generalize the Heisenberg group endowed with a left-invariant metric. A metric two-step nilpotent Lie algebra $\mf{h}$ is of \emph{Heisenberg type} if 
\begin{align*}
    j(Z)^2 = -|Z|^2 I_{\mf{v}}
\end{align*}
for every $Z \in \mf{z}$ (see \eqref{eq:j-map def}). A simply connected two-step nilpotent Lie group with left-invariant metric is of Heisenberg type if its metric Lie algebra is of Heisenberg type. It is often the case that theorems concerning the Heisenberg group endowed with a left-invariant metric, or their analogous formulations, are also true for Heisenberg type manifolds. For an example, see the results of \cite{gornetmast2000lengthspectrum}. This paradigm does not appear to extend to the setting of Heisenberg type manifolds endowed with a left-invariant magnetic field. In this section, we show by way of a simple example that the computation of the lengths of closed magnetic geodesics becomes significantly more complex for Heisenberg type manifolds.

Let $\mf{h} = \myspan \{ X_1, \ldots, X_4, Z_1, Z_2 \}$ and define a bracket structure by
\begin{align*}
    [X_1, X_2] = Z_1 \qquad [X_1, X_3] = Z_2 \qquad [X_2, X_4] = -Z_2 \qquad [X_3, X_4] = Z_1
\end{align*}
and extending by bilinearity and skew-symmetry to all of $\mf{h}$. Define the metric on $\mf{h}$ by declaring $\{ X_1, X_2,X_3, \allowbreak X_4, Z_1, Z_2 \}$ to be an orthonormal basis. It is straightforward to check that $\mf{h}$ is of Heisenberg type with two dimensional center $\mf{z} = \myspan\{ Z_1, Z_2 \}$. Let $\{ \alpha_1, \alpha_2,\alpha_3, \alpha_4, \zeta_1, \allowbreak \zeta_2 \}$ be the basis of $\mf{h}^*$ dual to $\{ X_1, X_2,X_3, X_4, Z_1, Z_2 \}$. Let $H$ be the simply connected Lie group with left-invariant metric with metric Lie algebra $\mf{h}$. As in Lemma \ref{lem:every exact left-invariant 2-form is d of something central}, any exact, left-invariant 2-form is of the form $\Omega = d(\zeta_m)$ for some $\zeta_m \in \mf{z}^*$. For simplicity we take as magnetic field $\Omega = d(B\zeta_1)$. For each $\gamma \in H$, we wish to understand the $\gamma$-periodic geodesics and the associated periods. 

Proceeding as in the Heisenberg case in Section \ref{sec:compact quotients}, let $p(t) = \sum a_i(t) \alpha_i + \sum c_i(t) \zeta_i$ be the integral curve of the Euler vector field on $\mf{h}^*$ with initial condition $p(0) = \sum u_i \alpha_i + \sum z_i \zeta_i$. Then the component functions are
\begin{align*}
a_1(t) &= u_1 \cos(\hat{z}t) + \left( \frac{-z_1 u_2 - z_2 u_3}{\hat{z}} \right) \sin(\hat{z}t) \\
a_2(t) &= u_2 \cos(\hat{z}t) + \left( \frac{z_1 u_1 + z_2 u_4}{\hat{z}} \right) \sin(\hat{z}t) \\
a_3(t) &= u_3 \cos(\hat{z}t) + \left( \frac{z_2 u_1 - z_1 u_4}{\hat{z}} \right) \sin(\hat{z}t) \\
a_4(t) &= u_4 \cos(\hat{z}t) + \left( \frac{-z_2 u_2 + z_1 u_3}{\hat{z}} \right) \sin(\hat{z}t) \\
c_1(t) &= z_1 \\
c_1(t) &= z_2
\end{align*}
where $\hat{z} = \sqrt{z_1^2 + z_2^2}$. Next, let $\sigma(t)$ be the magnetic geodesic through the identity determined by the integral curve $p(t)$. Hence $\sigma(t)$ solves $\sigma'(t) = dh_{p(t)}$, where $h : \mf{h}^* \to \bb{R}$ is the Hamiltonian (see \eqref{eq:reduced magnetic Hamiltonian}). Writing $\sigma(t) = \Exp({\bf X}(t) + {\bf Z}(t))$, where $\bfX(t) = \sum x_i(t) X_i$ and $\bfZ(t) = \sum z_i(t) Z_i$, we have on the one hand
under trivialization by left-multiplication,
\begin{align*}
\sigma'(t) = {\bf X}'(t) + {\bf Z}'(t) + \frac{1}{2}[{\bf X}'(t), {\bf X}(t)] = \sum x_i'(t) X_i + {\bf Z}'(t) + \frac{1}{2}[{\bf X}'(t), {\bf X}(t)],
\end{align*}
and on the other hand
\begin{align*}
dh_{p(t)} &= \sharp(p(t) + B\zeta_1) = \sum a_i(t)X_i + (z_1 + B)Z_1 + z_2 Z_2.
\end{align*}
Matching up the non-central components shows that
\begin{align*}
x_1(t) &= \frac{u_1}{\hat{z}} \sin(\hat{z}t)  + \frac{- z_1 u_2 - z_2 u_3}{\hat{z}^2}\left( 1 - \cos(\hat{z}t) \right) \\
x_2(t) &= \frac{u_2}{\hat{z}} \sin(\hat{z}t) + \frac{z_1 u_1 + z_2 u_4}{\hat{z}^2} \left( 1 - \cos(\hat{z}t) \right)\\
x_3(t) &= \frac{u_3}{\hat{z}} \sin(\hat{z}t) + \frac{z_2 u_1 - z_1 u_4}{\hat{z}^2}\left( 1 - \cos(\hat{z}t) \right) \\
x_4(t) &= \frac{u_4}{\hat{z}} \sin(\hat{z}t)  + \frac{- z_2 u_2 + z_1 u_3}{\hat{z}^2}\left( 1 - \cos(\hat{z}t) \right)
\end{align*}
while the central components satisfies
\begin{align*}
{\bf Z}'(t) = (z_1 + B)Z_1 + z_2 Z_2 - \frac{1}{2}[{\bf X}'(t), {\bf X}(t)].
\end{align*}
A tedious computation shows
\begin{align*}
\left[{\bf X}'(t), {\bf X}(t)\right] &= \left( - \frac{z_1 \hat{u}^2}{\hat{z}^2} + \frac{z_1 \hat{u}^2}{\hat{z}^2} \cos(\hat{z}t) \right)Z_1 + \left( - \frac{z_2 \hat{u}^2}{\hat{z}^2} + \frac{z_2 \hat{u}^2}{\hat{z}^2} \cos(\hat{z}t) \right)Z_2 \\
&= - \frac{z_1 \hat{u}^2}{\hat{z}^2} \left( 1 - \cos(\hat{z}t) \right)Z_1  - \frac{z_2 \hat{u}^2}{\hat{z}^2} \left( 1 - \cos(\hat{z}t) \right)Z_2 \\
&= \left( 1 - \cos(\hat{z}t) \right) \left( - \frac{z_1 \hat{u}^2}{\hat{z}^2} Z_1  - \frac{z_2 \hat{u}^2}{\hat{z}^2} Z_2 \right)
\end{align*}
where $\hat{u} = \sqrt{u_1^2 + \cdots + u_4^2}$. A final integration now provides the central components:
\begin{align*}
z_1(t) &= \left( z_1 + B + \frac{z_1 \hat{u}^2}{2\hat{z}^2} \right)t - \frac{z_1 \hat{u}^2}{2\hat{z}^3} \sin(\hat{z}t) \\
z_2(t) &= \left( z_2 + \frac{z_2 \hat{u}^2}{2\hat{z}^2} \right) t - \frac{z_2 \hat{u}^2}{2\hat{z}^3} \sin(\hat{z}t).
\end{align*}

Let $\gamma = \exp(\xi_1 Z_1 + \xi_2 Z_2)$ be a central element of $H$. Comparing the components of $\gamma \sigma(t) = \sigma(t + \omega)$ shows that $\omega = 2 \pi k / \hat{z}$, $k \in \bb{Z}$. With this choice of $\omega$, the non-central components are equal, while the central component yield the system
\begin{align}
    & \left( z_1 + B + \frac{z_1 \hat{u}^2}{2\hat{z}^2} \right) \frac{2\pi k}{\hat{z}} = \xi_1 \label{eq:HT central periods 1}\\
    & \left( z_2 + \frac{z_2 \hat{u}^2 }{2 \hat{z}^2} \right) \frac{2\pi k}{\hat{z}} = \xi_2. \label{eq:HT central periods 2}
\end{align}
Each choice of $u_1, \ldots, u_4, z_1, z_2$ satisfying this system and the energy constraint
\begin{align*}
    1 = \hat{u}^2 + (z_1 + B)^2 + z_2^2 = \hat{u}^2 + \hat{z}^2 + 2z_1 B + B^2
\end{align*}
will yield a unit speed magnetic geodesic translated by $\gamma$. In the case that the magnetic field and $\gamma$ are ``parallel'', i.e. $\xi_2 = 0$, then \eqref{eq:HT central periods 2} becomes
\begin{align*}
    z_2 \left(1 + \frac{\hat{u}^2 }{2 \hat{z}^2} \right) \frac{2\pi k}{\hat{z}} = 0.
\end{align*}
The second and third factor are necessarily nonzero, so $z_2 = 0$. This reduces \eqref{eq:HT central periods 1} to an equation that can be solved in the same way as the Heisenberg case, according to the strength of the magnetic field relative to the energy. When $\gamma$ is an arbitrary element of the center, it it is much more difficult to completely solve \eqref{eq:HT central periods 1} and \eqref{eq:HT central periods 2}, and hence obtain an explicit description of all the $\gamma$-periodic geodesics. 

\section{Appendix: Tangent Bundle Viewpoint: Periodic Magnetic Geodesics in Heisenberg manifolds}

In \cite{kaplan1981Cliffordmodules}, A. Kaplan introduced so-called $j$-maps to study Clifford modules (see Section \ref{sec:geoemtry of two-step nilpotent lie groups} for the definition). A metric two-step nilpotent Lie algebra is completely characterized by its associated $j$-maps. Since being introduced, they have proven very useful in the study of two-step nilpotent geometry. In this appendix we show how the magnetic geodesic equations can be characterized in terms of the $j$-maps.

Let $G$ be a two-step nilpotent Lie group endowed with a left-invariant metric $g$ and an exact, left-invariant magnetic form $\Omega$. Let $\mf{g} = \mf{v} \oplus \mf{z}$ be the decomposition of the Lie algebra into the center and its orthogonal complement. By Lemma \ref{lem:every exact left-invariant 2-form is d of something central}, there is $\zeta_m \in \mf{z}^*$ such that $\Omega = d(B\zeta_m)$.

\begin{lem} \label{lem:Lorenta force is j of something}
The Lorentz force associated to the magnetic field $\Omega$ satisfies $F_\mf{v} = j(-BZ_m)$ and $F_\mf{z} = 0$, where $Z_m = \sharp(\zeta_m)$.
\end{lem}

\begin{proof}
Let $X \in \mf{g}$, $V \in \mf{v}$ and $Z \in \mf{z}$. Then
\begin{align*}
    g(F(Z), X) &= \Omega(Z, X) = d(B \zeta_m)(Z,X) = -B\zeta_m([Z,X]) = 0,
\end{align*}
and
\begin{align*}
    g(F(V), X) &= \Omega(V, X) = d(B \zeta_m)(V,X) = -B\zeta_m([V,X]) \\
    &= -Bg(\sharp(\zeta_m), [V,X]) = -Bg(j(Z_m)V, X) \\
    &= g(j(-B Z_m)V, X).
\end{align*}
\end{proof}
Because of Lemma \ref{lem:Lorenta force is j of something}, we will write $F = j(-BZ_m)$ with the understanding that $F$ vanishes on central vectors and agrees with $j(-BZ_m)$ on vectors in $\mf{v}$. Let $\gamma(t) = \exp(X(t) + Z(t))$ be a magnetic geodesic on $G$ where $X(t) \in \mf{v}$ and $Z(t) \in \mf{z}$. By \eqref{eq:tangent vector field of path in two-step nilpotent group}, we can express the velocity vector of $\gamma$ as $\gamma'(t) = X'(t) + \frac{1}{2}[X'(t), X(t)] + Z'(t)$. The condition for $\gamma$ to be a magnetic geodesic is $\nabla_{\gamma'(t)} \gamma'(t) = F(\gamma'(t))$. Using \eqref{eq:Levi-Civita eqns} to expand this condition and imposing the initial conditions $\gamma(0) = e$ and $\gamma'(0) = X_0 + Z_0$, the geodesic equations on $\mf{v}$ and $\mf{z}$ separately are
\begin{align}
    & X''(t) = j(Z_0 - BZ_m)X'(t) \\
    & Z'(t) + \frac{1}{2}[X'(t),X(t)] = Z_0.
\end{align}

We restrict to the three-dimensional Heisenberg case and consider the magnetic geodesics in this context. Following the approach as illustrated in Prop. 3.5 on pages 625--628 of \cite{eberlein1994geometry}, and reducing to the three-dimensional Heisenberg case, a straightforward calculation gives   the following result.

\begin{cor}
\label{GeodEqsThm}If $z_{0}-B=0$, (or if $z_{0}-B\neq 0$ and $x_{0}=y_{0}=0$%
) then $\sigma \left( t\right) $ is the one parameter subgroup \ 
\begin{equation}
\sigma \left( t\right) =\exp \left( x\left( t\right) X+y\left( t\right)
Y+z\left( t\right) Z\right) =\exp \left( t\left( x_{0}X+y_{0}Y+z_{0}Z\right)
\right) \text{.}  \label{StraightGeodEqs}
\end{equation}%
If $z_{0}-B\neq 0$, the solution is 
\begin{equation}
\left( 
\begin{array}{c}
x\left( t\right) \\ 
y\left( t\right)%
\end{array}%
\right) =\frac{A}{z_{0}-B}\left( 
\begin{array}{cc}
\sin \left( t\left( \frac{z_{0}-B}{A}\right) \right) & -\left( 1-\cos \left(
t\left( \frac{z_{0}-B}{A}\right) \right) \right) \\ 
1-\cos \left( t\left( \frac{z_{0}-B}{A}\right) \right) & \sin \left( t\left( 
\frac{z_{0}-B}{A}\right) \right)%
\end{array}%
\right) \left( 
\begin{array}{c}
x_{0} \\ 
y_{0}%
\end{array}%
\right) \text{,}  \label{SpiralingGeodEqsNoncentral}
\end{equation}%
and%
\begin{equation}
\begin{array}{c}
 z\left( t\right) 
=\left( z_{0}+\frac{A\left(x_0^2+y_0^2\right)}{2\left( z_{0}-B\right) }\right) t-\frac{A^2%
\left(x_0^2+y_0^2\right)}{2\left( z_{0}-B\right) ^{2}}\sin \left( t\left( \frac{z_{0}-B}{A%
}\right) \right).  \label{SpiralingGeodEqsCentral}
\end{array}
\end{equation}
\end{cor}

\begin{rmk}
The coordinate functions \eqref{SpiralingGeodEqsNoncentral} and \eqref{SpiralingGeodEqsCentral} are equivalent the one obtained in \eqref{eq:x_i component of mag geod}-\eqref{eq:z component of mag geod} in the following sense. In order to obtain the magnetic geodesic through the origin determined by $(u_0, v_0, z_0)$ as in section \ref{sec:magnetic geodesic equations on simply connected 2n+1 dim Heis} take as initial tangent vector in \eqref{SpiralingGeodEqsNoncentral} and \eqref{SpiralingGeodEqsCentral} to be $(x_0, y_0, z_0) = (u_0/A, v_0/A, z_0 + B)$.
\end{rmk}

We now present some of the main results about the three-dimensional Heisenberg manifold proved in the body of the paper, but expressed using the tangent bundle, rather than the cotangent bundle.

Continuing the notation from the previous sections, we fix energy $E,$ magnetic strength $B,$ and metric parameter $A.$   
Let  $\left( H, g_A, \Omega\right)$ denote a simply connected Heisenberg manifold.  The theorems in this section state precisely the set of periods $\omega$ such that
there exists an intial velocity $v_{p}\in TH$ such that $\sigma_{v_{p}}\left(t\right)$
is periodic with period $\omega$. We also precisely state the
set of initial velocities $v_{p}$, hence the set of geodesics, that
produce each period $\omega$.  

Let $\Gamma$  denote a cocompact discrete subgroup of $H$ and, as above, denote the resulting  compact Heisenberg manifold by $\left(\Gamma\backslash H, g_A, \Omega\right).$ For all $\gamma \in \Gamma,$ we 
state below precisely the set of periods $\omega$ such that
there exists an intial velocity $v_{p}\in TH$ such that $\sigma_{v_{p}}\left(t\right)$
is $\gamma$-periodic with period $\omega$. We also precisely state the
set of initial velocities $v_{p}$, hence the set of geodesics, that
produce each period $\omega$.

\subsection{Periodic Magnetic Geodesics on the Simply Connected Heisenberg Group}

We now consider the existence
of periodic geodesics in $\left(H,g_{A}, d\left(B\zeta\right)\right)$,
the three-dimensional Heisenberg Lie group $H$ with left-invariant
metric determined by the orthonormal basis $\left\{ \frac{1}{\sqrt{A}}X,\frac{1}{\sqrt{A}}Y,Z\right\} $
and magnetic form $\Omega=-B\,\alpha\wedge\beta$. \ Recall that for
a vector $v\in\mathfrak{h}$, $\sigma_{v}\left(t\right)$ denotes
the magnetic geodesic through the identity with initial velocity $v$.
Note that if $v_{p}\in T_{p}H$, then $\sigma_{v_{p}}\left(t\right)$
denotes the magnetic geodesic through $p=\sigma_{v_{p}}\left(0\right)$
with initial velocity $v_{p}$. Also note that because $g_{A}$ and
$\Omega$ are left-invariant, that $\sigma_{v_{p}}\left(t\right)=L_{p}\sigma_{v}\left(t\right)$,
where $v_{p}=L_{p\ast}\left(v\right)$; i.e., magnetic geodesics through
$p\in H$ are just left translations of magnetic geodesics through
the identity. Clearly, a magnetic geodesic through $p\in H$ is periodic
with period $\omega$ if and only if its left translation by $p^{-1}$
is a magnetic geodesic through the identity with period $\omega$.

\begin{thm}
\label{thm:PreciseThmSimplyConnected}With notation as above, fix energy
$E,$ magnetic strength $B,$ and metric paramter $A.$
\begin{enumerate}
\item If $B^2>E^2$ then there exists a one-parameter
family of vectors $v\in\mathfrak{h}$, $\left\vert v\right\vert =E,$
such that $\sigma_{v}\left(t\right)$ is periodic. In particular,
$\sigma_{v}\left(t\right)$ is periodic if and only if $z_{0}=B-\mathrm{sgn}\left(B\right)\sqrt{B^{2}-E^{2}}$and
\[
x_{0}^{2}+y_{0}^{2}=-2z_{0}\left(z_{0}-B\right)/A.
\]
The set of periods of $\sigma_{v}\left(t\right)$ is $\frac{2\pi}{\sqrt{B^{2}-E^{2}}}\mathbb{Z}_{\neq0}$
and the smallest positive period is $\omega=\left\vert \frac{2\pi A}{z_{0}-B}\right\vert =\frac{2\pi A}{\sqrt{B^{2}-E^{2}}}$.
\item If $B^2\leq E^2$ then there does not exist a
vector $v$ with $\left\vert v\right\vert =E$ such that $\sigma_{v}\left(t\right)$
is periodic. 
\end{enumerate}
\end{thm}

\subsection{Periodic Geodesics on Compact Quotients of the Heisenberg group}

We ultimately wish to consider closed magnetic geodesics on Heisenberg
manifolds of the form $\Gamma\backslash H$, where $\Gamma$ is a
cocompact discrete subgroup of $H.$ As above, we proceed by considering
$\gamma$-periodic magnetic geodesics on the cover $H$.

The purpose of this section is stating more precisely, and proving,
the following, which is divided into several cases. See Theorem \ref{thm:NonCentralThm},
Theorem \ref{thm:StraightThm}, and Theorem \ref{thm:SpiralingThm} below.

\begin{thm}
Consider the three-dimensional Heisenberg Lie group $H$ with left
invariant metric $g_{A}$ determined by the orthonormal basis $\left\{ \frac{1}{\sqrt{A}}X,\frac{1}{\sqrt{A}}Y,Z\right\} $
and magnetic form $\Omega=d\left(B\zeta\right)=-B\alpha\wedge\beta$. Fix $\gamma\in H$
and fix energy $E,$ magnetic strength $B$ and metric parameter $A.$
Then we can state precisely the set of periods $\omega$ such that
there exists an initial velocity $v_{p}\in TH$ with $\left\vert v_{p}\right\vert =E$
such that $\sigma_{v_{p}}\left(t\right)$ is $\gamma$-periodic with
period $\omega$. We can also precisely state the set of initial velocities
$v_{p}$, hence the set of geodesics, that produce each period $\omega$.
\end{thm}



\subsubsection{Noncentral Case}

Let $\gamma=\exp\left(x_\gamma X+y_\gamma Y+z_\gamma Z\right)\in H$ with
$x_\gamma^{2}+y_\gamma^{2}\neq0$. Let $a=\exp\left(a_{x}X+a_{y}Y+a_{z}Z\right)\in H$.\ From
(\ref{eq:multiplication law}), 
the conjugacy class of $\gamma$ in $H$ is $\exp\left(\hat{x}X+\hat{y}Y+\mathbb{R}Z\right)$.
\begin{thm}
\label{thm:NonCentralThm} Fix energy $E,$ magnetic strength $B$ and metric parameter $A.$
Let $\gamma=\exp\left(x_\gamma X+y_\gamma Y+z_\gamma Z\right)\in H$ with
$x_\gamma^{2}+y_\gamma^{2}\neq0$.
\begin{enumerate}
\item If $E^2>B^2$ (ie, if $\mu>1$ ) then there exists
a two-parameter family of elements $a\in H$ such that $a\gamma a^{-1}=\exp\left(x_\gamma X+y_\gamma Y+z_\gamma^{\prime}Z\right)$
where $z_\gamma^{\prime}=\pm B\sqrt{A\left(x_\gamma^{2}+y_\gamma^{2}\right)/\left(E^{2}-B^{2}\right)}$.
Letting $v=$ $\frac{B}{z_\gamma^{\prime}}\left(x_\gamma X+y_\gamma Y+z_\gamma^{\prime}Z\right)$,
which satisfies $\left\vert v\right\vert =E$, then $\gamma$ translates
the (non-spiraling) magnetic geodesic $a^{-1}\exp\left(tv\right)$
with period $\omega=\pm\sqrt{A\left(x_\gamma^{2}+y_\gamma^{2}\right)/\left(E^{2}-B^{2}\right)}$.
These are the only magnetic geodesics with energy $E$ translated
by $\gamma$.
\item If $E^2\leq B^2$ (ie, if $\mu\leq1$ ) then neither
$\gamma$ nor any of its conjugates in $H$ translate a magnetic geodesic
with energy $E$. 
\end{enumerate}
\end{thm}

\subsubsection{Central Case}

Throughout this subsection, we assume that $x_\gamma=y_\gamma=0$;
i.e., that $\gamma$ lies in $Z\left(H\right),$ the center of three-dimensional
Heisenberg group $H.$ Recall that since $\gamma$ is central, $\gamma$
translates a magnetic geodesic $\sigma\left(t\right)$ through the
identity $e\in H$ with period $\omega$ if and only if for all $a\in H$,
$\gamma$ translates a magnetic geodesic through $a$ with period
$\omega$. \ That is, without loss of generality, if $\gamma$ lies
in the center, we may assume that $\sigma\left(0\right)=e$. 

Recall that magnetic geodesics in $H$ are either spiraling or one
parameter subgroups. We first consider the case of one-parameter subgroups.
\begin{thm}
\label{thm:StraightThm} Fix energy $E,$ magnetic strength $B$ and metric parameter $A.$
\ Let $\gamma=\exp\left(z_\gamma Z\right)\in Z\left(H\right)$, with
$z_\gamma \neq0$. The element $\gamma$ translates the magnetic geodesics
$\sigma\left(t\right)=$ $\exp\left(\pm tEZ\right)$ with initial
velocities $v=\pm EZ$ and periods $\omega=\pm z_\gamma /E$. This pair
of one-parameter subgroups and their left translates are the only
straight magnetic geodesics translated by $\gamma$.
\end{thm}
%
%

\begin{thm}
\label{thm:SpiralingThm}Fix energy $E,$ magnetic strength $B, $ and metric parameter $A.$
Denote $\mu=\frac{E}{\left\vert B\right\vert }$. Let $\gamma=\exp\left(z_\gamma Z\right)\in Z\left(H\right)$,
$z_\gamma\neq0$. If there exists a vector $v=x_{0}X+y_{0}Y+z_{0}Z$
and a period $\omega\neq0$ such that the spiraling geodesic $\sigma_{v}\left(t\right)$
is $\gamma$-periodic with period $\omega$, then there exists $\ell\in\mathbb{Z}_{\neq0}$
such that $\zeta_{\ell}=$ $\frac{z_\gamma}{\pi \ell}$ $\ $satisfies the
conditions relative to $\mu$ specified in the following six cases
and $A\left(x_{0}^{2}+y_{0}^{2}\right)$, $z_{0}$, and
$\omega$ are as expressed below. Conversely, for every choice of
$\ell\in\mathbb{Z}_{\neq0}$ $\ $such that $\zeta_{\ell}=\frac{z_\gamma}{\pi \ell}$
satisfies the conditions in one of the cases below, there exists at
least one vector $v$ as given below such that $\sigma_{v}\left(t\right)$
is $\gamma$-periodic (spiraling) geodesic with period $\omega$ as
given below. \ Note that Case 1 requires $E^2<B^2$.
Cases 2 through 5 require $E^2>B^2,$ and Case
6 requires $E^2=B^2.$ Note that in all cases, $\zeta_\ell\neq0.$

\begin{enumerate}
\item $\frac{-2\mu}{1-\mu}<\frac{\zeta_{\ell}}{A}<\frac{2\mu}{1+\mu}<1$,
\item $1<\frac{2\mu}{1+\mu}<\frac{\zeta_{\ell}}{A}<2$,
\item $2<\frac{\zeta_{\ell}}{A}\leq\frac{2\mu}{\mu-1}$,
\item $2<\frac{2\mu}{\mu-1}<\frac{\zeta_{\ell}}{A}$,
\item $\frac{\zeta_{\ell}}{A}=2$ and $\mu>1$,
\item $\frac{\zeta_{\ell}}{A}=1$ and $\mu=1.$ 
\end{enumerate}
In Cases 1 through 4, we choose any $x_{0},y_{0}\in\mathbb{R}$ so
that 
\begin{equation}
A\left(x_0^2+y_0^2\right)=B^{2}\left(\frac{\mu^{2}-1}{\frac{\zeta_{\ell}}{A}-1}\left(\frac{\zeta_{\ell}}{A}-2\right)+2\sqrt{\frac{\mu^{2}-1}{\frac{\zeta_{\ell}}{A}-1}}\right)\label{Ebar-plus}
\end{equation}
and let
\begin{equation}
z_{0}=-B\left(-1+\sqrt{\frac{\mu^{2}-1}{\frac{\zeta_{\ell}}{A}-1}}\right)\label{z0-plus}
\end{equation}
and 
\[
\omega=\frac{2z_\gamma A}{\zeta_{\ell}\left(z_{0}-B\right)}=\frac{2\pi \ell\sqrt{\left\vert \frac{\zeta_{\ell}}{A}-1\right\vert }}{\sqrt{E^{2}-B^{2}}}\text{.}
\]
In Case 4, we may also choose any $x_{0},y_{0}\in\mathbb{R}$ so that
\begin{equation}
A\left(x_0^2+y_0^2\right)=B^{2}\left(\frac{\mu^{2}-1}{\frac{\zeta_{\ell}}{A}-1}\left(\frac{\zeta_{\ell}}{A}-2\right)-2\sqrt{\frac{\mu^{2}-1}{\frac{\zeta_{\ell}}{A}-1}}\right)\label{Ebar-minus}
\end{equation}
and let 
\begin{equation}
z_{0}=-B\left(-1-\sqrt{\frac{\mu^{2}-1}{\frac{\zeta_{\ell}}{A}-1}}\right)\label{z0-minus}
\end{equation}
and 
\[
\omega=\frac{2 z_\gamma A}{\zeta_{\ell}\left(z_{0}-B\right)}=-\frac{2\pi \ell\sqrt{\left\vert \frac{\zeta_{\ell}}{A}-1\right\vert }}{\sqrt{E^{2}-B^{2}}}\text{.}
\]
The conditions on $\mu,\zeta_{\ell},x_{0}$,$y_{0\text{ }}$and $z_{0}$
imply $\frac{\mu^{2}-1}{\frac{\zeta_{\ell}}{A}-1}>0$, $x_0^2+y_0^2>0$,
$E^{2}=A\left(x_0^2+y_0^2\right) \allowbreak +z_{0}^{2}$, and the (spiraling) magnetic geodesic
through the identity $\sigma_{v}\left(t\right)$ with initial velocity
$v=x_{0}X+y_{0}Y+z_{0}Z$ is $\gamma=\exp\left(z_\gamma Z\right)$-periodic
with energy $E$ and period $\omega$ as given.

In Case 5, which only occurs if $\frac{z_\gamma}{A}\in2\pi\mathbb{Z}_{\neq0}$,
we choose any $x_{0},y_{0\text{ }}\in\mathbb{R}$ so that 
\[
A\left(x_0^2+y_0^2\right)=2\left\vert B\right\vert \sqrt{E^{2}-B^{2}}
\]
and
\[
z_{0}=B-\frac{A\left(x_0^2+y_0^2\right)}{2B}\text{.}
\]
Then the conditions on $\mu,\zeta_{\ell},x_{0},y_{0}$ and $z_{0\text{ }}$imply
that $E^{2}=A\left(x_0^2+y_0^2\right)+z_{0}^{2}$ and the (spiraling) magnetic geodesic
$\sigma_{v}\left(t\right)$ starting at the identity with initial
velocity $v=x_{0}X+y_{0}Y+z_{0}Z$ is $\gamma$-periodic with energy
$E$ and period 
\[
\omega=-\mathrm{sgn}\left(B\right)\frac{z_\gamma} {\sqrt{E^{2}-B^{2}}}\text{.}
\]

In Case 6, which only occurs if $\frac{z_\gamma}{A}\in\pi\mathbb{Z}_{\neq0}$,
we choose any $x_{0},y_{0\text{ }},z_{0}\in\mathbb{R}$ so that $E^{2}=B^{2}=A\left(x_{0}^{2}+y_{0}^{2}\right)+z_{0}^{2}$
and $z_{0}\neq\pm B$. The conditions on $\mu$ and $\zeta_{\ell}$ imply
that the (spiraling) magnetic geodesic $\sigma_{v}\left(t\right)$
with intial velocity $v=x_{0}X+y_{0}Y+z_{0}Z$\ will yield a $\gamma$-periodic
magnetic geodesic with energy $E$ and period 
\[
\omega=\frac{2 z_\gamma} {z_{0}-B}\text{.}
\]
\end{thm}
\begin{rmk}
In Case 1, there are infinitely many values of $\ell$ that satisfy the
conditions, hence infinitely many distinct periods $\omega.$ In particular,
if $\mu<1$ and there exists $\ell_{0}\in\mathbb{Z}_{>0}$ such that
$\zeta_{\ell_{0}}\in\left(\frac{-2\mu}{1-\mu},\frac{2\mu}{1+\mu}\right)$,
then for all $\ell>\ell_{0}$, $\zeta_{\ell}\in\left(\frac{-2\mu}{1-\mu},\frac{2\mu}{1+\mu}\right)$.
Likewise if there exists $\ell_{0}\in\mathbb{Z}_{<0}$ such that $\zeta_{\ell_{0}}\in\left(\frac{-2\mu}{1-\mu},\frac{2\mu}{1+\mu}\right)$,
then for all $\ell<\ell_{0}$, $\zeta_{\ell}\in\left(\frac{-2\mu}{1-\mu},\frac{2\mu}{1+\mu}\right)$.
\end{rmk}

\begin{rmk}
In Case 6, the magnitude of the periods take all values in the interval
$\left(\left\vert z_\gamma\right\vert /E,\infty\right)$. The period
$\omega=\left\vert z_\gamma\right\vert /E$ is achieved when $v=-BZ$,
which implies $\sigma_{v}$is a one-parameter subgroup; i.e., non-spiraling.
The magnitude of the period approaches $\infty$ as $v\rightarrow BZ$.
This behavior is in contrast to the Riemannian case; i.e., the case
$B=0$. In the Riemannian case, there are finitely many periods associated
to each element $\gamma.$ However, if there exists $\gamma\in\Gamma$
such that $\log\gamma\in2\pi\mathbb{Z},$ then $\Gamma\backslash H$
does not satisfy the Clean Intersection Hypothesis, so the fact that
unusual magnetic geodesic behavior occurs in this case is not unprecedented (see \cite{gornet2005trace}).
\end{rmk}

\bibliographystyle{alpha}
\bibliography{mag-flow}

\end{document}